%% file: main.tex
\documentclass[a4paper,11pt]{article}
\usepackage[top=2.5cm, bottom=2.5cm, left=2.5cm, right=2.5cm]{geometry}
\usepackage{settings}
\allowdisplaybreaks

\begin{document}
\title{Average-case matrix discrepancy: satisfiability bounds}
\date{}
\author{Antoine Maillard\thanks{\texttt{antoine.maillard@inria.fr}.}}
\affil{Inria Paris \& DIENS, Ecole Normale Sup\'erieure, PSL Research University. 
\\ Department of Mathematics, ETH Z\"urich, Switzerland.}

\maketitle

\begin{abstract}
    Given a sequence of $d \times d$ symmetric matrices $\{\bW_i\}_{i=1}^n$, and a margin $\Delta > 0$, we investigate whether it is possible to find signs $(\eps_1, \dots, \eps_n) \in \{\pm 1\}^n$ such that the operator norm of the signed sum satisfies $\|\sum_{i=1}^n \eps_i \bW_i\|_\op \leq \Delta$.
    \cite{kunisky2023online} recently introduced a random version of this problem, where the matrices $\{\bW_i\}_{i=1}^n$ are drawn from the Gaussian orthogonal ensemble. 
    This model can be seen as a random variant of the celebrated Matrix Spencer conjecture, and as a matrix-valued analog of the symmetric binary perceptron in statistical physics.
    In this work, we establish a satisfiability transition in this problem as $n, d \to \infty$ with $n / d^2 \to \tau > 0$.
    Our main results are twofold.
    First, we prove that the expected number of solutions with margin $\Delta = \kappa \sqrt{n}$ has a sharp threshold at a critical $\tau_1(\kappa)$: 
    for $\tau < \tau_1(\kappa)$ the problem is typically unsatisfiable, while for $\tau > \tau_1(\kappa)$ the average number of solutions becomes exponentially large.
    Second, combining a second-moment method with recent results from \cite{altschuler2023zero} on margin concentration in perceptron-type problems, we identify a second threshold $\tau_2(\kappa)$, 
    such that for $\tau > \tau_2(\kappa)$ the problem admits solutions with high probability.
    In particular, we establish that a system of $n = \Theta(d^2)$ Gaussian random matrices can be balanced so that the spectrum of the resulting matrix macroscopically shrinks compared to the typical semicircle law.
    Finally, under a technical assumption, we show that there exists values of $(\tau,\kappa)$ for which the number of solutions has large variance, 
    implying the failure of the second moment method and uncovering a richer picture than in the vector-analog symmetric binary perceptron problem.
    Our proofs rely on concentration inequalities and large deviation properties for the law of correlated Gaussian matrices under spectral norm constraints.
\end{abstract}

\newpage
\setcounter{tocdepth}{2}
\tableofcontents

\newpage
\section{Introduction and main results}\label{sec:intro}
\input{sections/introduction.tex}

\subsection*{Acknowledgements}
The author wants to thank Dylan Altschuler for discussions and a careful reading of a first version of this manuscript, 
Tim Kunisky and Afonso S.\ Bandeira for insightful discussions about the early stages of this work, 
and anonymous reviewers for helpful comments and suggestions.

\section{Sharp asymptotics of the first moment}\label{sec:1st_moment}
\input{sections/first_moment.tex}

\section{The satisfiability region}\label{sec:2nd_moment}
\input{sections/second_moment.tex}

\section{Limitations of the second moment method}\label{sec:fail_2nd_moment}
\input{sections/fail_2nd_moment.tex}

\printbibliography

\newpage
\appendix 
\addtocontents{toc}{\protect\setcounter{tocdepth}{1}} 

\section{Contiguity with a planted model, and freezing of solutions}\label{sec:appendix_freezing}
\input{sections/appendix_freezing.tex}

\section{Large deviations for the spectral norm of a Gaussian matrix}\label{sec:appendix_ldp}
\input{sections/appendix_ldp.tex}

\end{document}

%% file: sections/introduction.tex
\subsection{Setting and related literature}

We start by introducing in Sections~\ref{subsubsec:discrepancy} and \ref{subsubsec:sbp} two sets of independent definitions and models,
both being important motivations behind the problem we study, which is discussed in Section~\ref{subsubsec:average_mdiscrepancy}.

\subsubsection{Discrepancy theory}\label{subsubsec:discrepancy}

Computing the \emph{discrepancy} (i.e.\ the best possible balancing) of a collection of sets or vectors 
is a classical question in mathematics and theoretical computer science. 
It has wide-ranging applications in combinatorics, computational geometry,
experimental design, and the theory of approximation algorithms, to name a few. 
The reader will find in \cite{spencer1994ten,matousek2009geometric,chen2014panorama} a detailed account of the history and applications of 
discrepancy theory.
Arguably one of the most celebrated results in this area is the following theorem, known as Spencer's ``six deviations suffice''~\citep{spencer1985six}.
\begin{theorem}[\cite{spencer1985six}]
    \label{thm:spencer}
    There exists $C > 0$ such that for all $n, d \geq 1$, and all $\bu_1, \cdots, \bu_n \in \bbR^d$ with $\|\bu_i\|_\infty \leq 1$ for all $i \in [n]$: 
    \begin{equation*}
       \disc(\bu_1, \cdots, \bu_n) \coloneqq \min_{\eps \in \{\pm 1\}^n} \left\|\sum_{i=1}^n \eps_i \bu_i\right\|_\infty \leq
       \begin{dcases}
       C \sqrt{n \max \left(1, \log \frac{d}{n}\right)} &\textrm{ if } n \leq d, \\
       C \sqrt{d} &\textrm{ if } n > d.
       \end{dcases}
    \end{equation*}
\end{theorem}
\noindent
The second case $n > d$ can be deduced from the result for $n = d$ using classical arguments based on iterated rounding, see e.g.~\cite{spencer1994ten}.
Interestingly, Spencer's theorem shows that one can drastically improve over a naive pick of random signings $\eps_1, \cdots, \eps_n \iid \Unif(\{\pm 1\})$. Indeed, 
it is not hard to show that (taking $n = d$ for simplicity) random signings achieve $\left\|\sum_{i=1}^n \eps_i \bu_i\right\|_\infty = \Theta(\sqrt{n \log n})$ for a worst-case choice of $(\bu_i)_{i=1}^n$, 
with high probability as $n \to \infty$. Spencer's theorem thus implies the existence of a signing $\eps \in \{\pm 1\}^n$, 
which depends on the value of the $\bu_i$'s, and whose discrepancy generically improves over the one of random signs by a logarithmic factor.
While the original proof of \cite{spencer1985six} is not constructive, there has been recently a great number of results regarding efficient algorithmic constructions of 
these signings. We will discuss this point further in Section~\ref{subsec:discussion}.

\myskip 
\textbf{Matrix discrepancy --}
In the present work we consider the problem of \emph{matrix discrepancy}. Given a set of $n$ symmetric $d \times d$ matrices $\bA_1, \cdots, \bA_n$, 
we aim to characterize the following discrepancy objective ($\|\bA\|_\op \coloneqq \max_{\|\bx\|_2=1} \|\bA \bx\|_2$ is the spectral, or operator, norm):
\begin{equation}
    \label{eq:def_disc}
    \disc(\bA_1, \cdots, \bA_n) \coloneqq \min_{\eps \in \{\pm 1\}^n} \left\| \sum_{i=1}^n \eps_i \bA_i \right\|_\op.
\end{equation}
As already mentioned, a foundational question in discrepancy is how the discrepancy objective compares to the one achieved by a \emph{random} choice of the 
signings $\eps_i \iid \Unif(\{\pm 1\})$: matrix discrepancy is thus intimately connected to the study of large random matrices.
Matrix discrepancy has also
been shown to have implications in the theory of quantum random access codes~\citep{hopkins2022matrix,bansal2023resolving}, 
in generalizations of the Kadison-Singer problem~\citep{marcus2015interlacing,kyng2020four}, as well as 
graph sparsification~\citep{batson2014twice} to name a few.

\myskip
The celebrated ``Matrix Spencer'' conjecture~\citep{zouzias2012matrix,meka2014discrepancy1} is likely the most important open problem in matrix discrepancy.
It asserts that Spencer's Theorem~\ref{thm:spencer} can be generalized to the matrix setting, as follows.
\begin{conjecture}[Matrix Spencer~\citep{zouzias2012matrix,meka2014discrepancy1}]\label{conj:mspencer}
    There exists $C > 0$ such that for all $n, d \geq 1$, and all $\bA_1, \cdots, \bA_n$ symmetric $d \times d$ matrices with $\max_{i \in [n]}\|\bA_i\|_\op \leq 1$:
    \begin{equation*}
        \disc(\bA_1, \cdots, \bA_n) \leq C \sqrt{n \max\left(1, \log \frac{d}{n}\right)}.
    \end{equation*}
    In particular, the discrepancy is $\mcO(\sqrt{n})$ for $d = n$.
\end{conjecture}
\noindent
We stress that Spencer's theorem can be seen as the special case of Conjecture~\ref{conj:mspencer} in which all $\bA_i$ commute with each other (and are thus diagonalizable in the same basis).
Moreover, a weaker form of Conjecture~\ref{conj:mspencer}, with a bound $\mcO(\sqrt{n \log d})$ on the right-hand side, can easily be shown to be achievable using a random choice 
of signs $\eps_i \iid \Unif(\{\pm 1\})$, using the non-commutative Khintchine inequality of~\cite{lust1991non}.
Despite a recent surge in efforts~\citep{levy2017deterministic,hopkins2022matrix,dadush2022new}, Conjecture~\ref{conj:mspencer} remains open at the time of this writing. 
The best-known result is a proof of Matrix Spencer if we additionally assume
$\rk(\bA_i) \lesssim n / \log^3 n$~\citep{bansal2023resolving}, and is based on the recent improvements over the non-commutative Khintchine inequality of \cite{bandeira2023matrix}.
In this work, we consider an average-case version of Conjecture~\ref{conj:mspencer}, introduced in Section~\ref{subsubsec:average_mdiscrepancy}.

\subsubsection{Random vector discrepancy, and the symmetric binary perceptron}\label{subsubsec:sbp}

\noindent
A natural question in vector discrepancy (i.e.\ in the setting of Theorem~\ref{thm:spencer}) is to shift our attention to average-case settings, where the $\bu_i$ are 
chosen randomly rather than potentially adversarially.
The typical discrepancy for random vectors is now well understood: in the regime $n = \omega(d)$ (i.e.\ $d = \smallO(n)$, many more signs than dimensions), 
\cite{turner2020balancing} established the following result (with partial results preceding in~\cite{karmarkar1986probabilistic,costello2009balancing}):
\begin{theorem}[\cite{turner2020balancing}]\label{thm:rvector_disc}
    Assume that $\bu_1, \cdots, \bu_n \iid \mcN(0, \Id_d)$ and that $n = \omega(d)$.
    Then  
    \begin{align*}
        \plim_{d \to \infty} \frac{\disc(\bu_1, \cdots, \bu_n)}{\sqrt{\frac{\pi n}{2}} 2^{-n/d}} = 1,
    \end{align*}
    where the limit is meant in probability.
\end{theorem}
\noindent
While Spencer's Theorem~\ref{thm:spencer} does not apply to such random vectors (as one can easily show that $\|\bu_i\|_\infty \sim \sqrt{2 \log d} \gg 1$),
the randomness makes the problem amenable to a detailed mathematical analysis with different tools.

\myskip
In a complementary way, 
recent works have characterized
very precisely the minimal discrepancy in the proportional regime $n = \Theta(d)$, as a function of the ``aspect ratio'' $\beta \coloneqq \lim_{d \to \infty} n/d$.
This setting of random vector discrepancy is an instance of the \emph{symmetric binary perceptron} (SBP), a random constraint satisfaction problem 
which was introduced in~\cite{aubin2019storage}
as a variant to the classical asymmetric binary perceptron. 
The latter is  
a simple model of a neural network storing random patterns, which has a long history of study in 
computer science, statistical physics, and probability theory~\citep{cover1965geometrical,gardner1988space,gardner1988optimal,krauth1989storage,sompolinsky1990learning,talagrand1999intersecting,talagrand2010mean}.
For a \emph{margin} $K > 0$, and in the limit $n/d \to \beta > 0$, the question of satisfiability in the SBP was introduced in~\cite{aubin2019storage} 
as:
\begin{center}
   \textit{
    Given $\bg_1, \cdots, \bg_d \iid \mcN(0, \Id_n)$, can we find $\eps \in \{\pm 1\}^n$ 
    such that $\max_{i \in [d]}|\langle \bg_i, \eps \rangle| \leq K \sqrt{n}$? 
   }
\end{center}
Letting $(\bu_i)_k \coloneqq (\bg_j)_k$, so that $(\bu_i)_{i=1}^n \iid \mcN(0, \Id_d)$ it is clear that the 
the SBP can be thought as an average-case version of vector discrepancy, 
and more precisely to the setting of Theorem~\ref{thm:rvector_disc} with $n = \Theta(d)$.
The SBP has received significant attention in the recent literature:
its relative simplicity allows studying in detail the relation between its structural properties and the performance of 
solving algorithms, 
which can then guide the theory of more complex statistical models,
see e.g.\ \cite{aubin2019storage,perkins2021frozen,abbe2022proof,gamarnik2022algorithms,kizildag2023symmetric,barbier2024atypical,alaoui2024hardness,barbier2024escape}.
Concretely, it is shown in \cite{aubin2019storage,abbe2022proof,gamarnik2022algorithms} (among other results) that 
the SBP undergoes the following sharp satisfiability/unsatisfiability transition.
\begin{theorem}[Sharp threshold for the SBP~\citep{aubin2019storage,abbe2022proof}]\label{thm:trans_SBP}
    Let $K > 0$.
    Define
    \begin{equation*}
        \beta_1(K) \coloneqq - \frac{\log \bbP_{z \sim \mcN(0,1)}[|z| \leq K]}{\log 2}.
    \end{equation*}
    For $\bu_1, \cdots, \bu_n \iid \mcN(0, \Id_d)$,
    let 
    \begin{equation*}
        Z_K \coloneqq \left|\left\{\eps \in \{\pm 1\}^n \, : \, \left\|\sum_{i=1}^n \eps_i \bu_i \right\|_\infty \leq K \sqrt{n}\right\}\right|.
    \end{equation*}
    Then in the limit $n, d \to \infty$ with $n / d \to \beta > 0$:
    \begin{itemize}
        \item[$(i)$] If $\beta < \beta_1(K)$, $\lim_{d \to \infty} \bbP[Z_K \geq 1] = 0$. 
        \item[$(ii)$] If $\beta > \beta_1(K)$, $\lim_{d \to \infty} \bbP[Z_K \geq 1] = 1$. 
    \end{itemize}
\end{theorem} 
\noindent
The analysis of \cite{aubin2019storage} is based on the second moment method, and has been refined in subsequent works~\citep{abbe2022proof,gamarnik2022algorithms}. In particular, 
$\beta_1(K)$ can be derived as the critical value of $\beta$ where the value of $\EE[Z_K]$ 
transitions from being exponentially small (in $n,d$) to exponentially large.

\subsubsection{Average-case matrix discrepancy}\label{subsubsec:average_mdiscrepancy}

In this work, 
alongside \cite{kunisky2023online},
we initiate the study of the average-case matrix discrepancy problem.
Concretely, given $n, d \geq 1$ a \emph{margin} $\kappa > 0$, and $\bW_1, \cdots, \bW_n$ random independent matrices, drawn with centered Gaussian i.i.d.\ elements (up to symmetry) with variance $1/d$, we seek to answer the question: 
\begin{center}
   \textrm{\textbf{(P)}} : \textit{Can we find signs $(\eps_1, \cdots, \eps_n) \in \{\pm 1\}^n$ such that $\|\sum_{i=1}^n \eps_i \bW_i\|_\op \leq \kappa \sqrt{n}$?}
\end{center}
\noindent 
This problem can be seen as an average-case analog of Conjecture~\ref{conj:mspencer}, with the simple Gaussian random matrix model serving as a natural starting point.
By investigating this simplified case in great detail, we firstly aim to gain insight, 
in order to possibly probe Conjecture~\ref{conj:mspencer} in more structured random matrix models in the future.
This would form an alternative direction to other recent works~\citep{hopkins2022matrix,dadush2022new,bansal2023resolving} towards a better understanding of the Matrix Spencer conjecture.

\myskip 
Additionally, \textbf{(P)} can naturally be thought of as the matrix discrepancy analog to the random vector discrepancy problem (or the symmetric binary perceptron)
described above:
a mapping can be achieved by modifying our model to
diagonal matrices $\bW_i$ with i.i.d.\ elements drawn from $\mcN(0,1)$ on the diagonal, 
and setting $(\bg_i)_k = (\bW_{i})_{kk}$.

\myskip
Importantly, we will focus on studying \textbf{(P)} in the case of a \emph{finite margin} $\kappa > 0$ as $n, d \to \infty$. 
In random vector discrepancy, the satisfiability transition for a finite margin occurs in the scale $n = \Theta(d)$, -- i.e.\ in the symmetric binary perceptron setting -- see Theorem~\ref{thm:trans_SBP}.
As we will see, in our model the critical scaling is rather $n = \Theta(d^2)$, and we will focus on this 
regime throughout our work.
Our approach to tackle \textbf{(P)} further builds on the methods introduced by \cite{aubin2019storage} for the SBP model.
Ultimately, the main goal we pursue is to obtain a detailed understanding of the satisfiability properties of \textbf{(P)}, 
and to reach a counterpart to Theorem~\ref{thm:trans_SBP} in the context of average-case matrix discrepancy.

\myskip 
\textbf{Parallel work --} 
Shortly after a first pre-print of the present manuscript was made available online, an independent study~\citep{wengiel2024asymptotic} appeared, exploring as well the asymptotic discrepancy of Gaussian i.i.d.\ matrices.
Their findings extend the approach of~\cite{turner2020balancing} for random vector discrepancy to this context,
by proving a counterpart to Theorem~\ref{thm:rvector_disc} for random matrix discrepancy.
Their results provide a sharp characterization in the regime $n = \omega(d^2)$ (i.e.\ $d = \smallO(\sqrt{n})$), and $\kappa = \smallO(1)$.
Technically, \cite{wengiel2024asymptotic} employs the second moment method, similar to our work, 
but relies on a general upper bound for the density of the joint law of the spectra of two correlated Gaussian matrices.
In contrast, the present work focuses on the critical regime $n = \Theta(d^2)$, where the asymptotic margin $\kappa > 0$ is finite.
In this critical regime, the bound used in~\cite{wengiel2024asymptotic} becomes vacuous, and we employ here different ideas to control this joint law.
Combining the results of~\cite{wengiel2024asymptotic} with ours, 
we achieve a nearly complete description of the asymptotic discrepancy of Gaussian i.i.d.\ matrices, 
leaving open only a fraction of the phase diagram in the critical regime $n = \Theta(d^2)$, see Fig.~\ref{fig:phase_diagram}.

\subsection{Main results}

\subsubsection*{Notations and background}

We denote $[d] \coloneqq \{1, \cdots, d\}$ the set of integers from $1$ to $d$,
and $\mcS_d$ the set of $d \times d$ real symmetric matrices.
For a function $V : \bbR \to \bbR$, and $\bS \in \mcS_d$ with eigenvalues $(\lambda_i)_{i=1}^d$, we define $V(\bS)$ as the matrix with the same eigenvectors as $\bS$,
and eigenvalues $(V(\lambda_i))_{i=1}^d$. 
We denote by $\|\bS\|_{\op} \coloneqq \max_{i \in [d]}|\lambda_i|$ the operator norm.
For any $B \subseteq \bbR$ we denote by $\mcM_1^+(B)$ the set of real probability distributions on $B$.
For a probability measure $\mu \in \mcM_1^+(\bbR)$ we denote $\Sigma(\mu) \coloneqq \int \mu(\rd x) \mu(\rd y) \log |x - y|$ its non-commutative entropy. 
We say that a random matrix $\bY \in \mcS_d$ is generated from the \emph{Gaussian Orthogonal Ensemble} $\GOE(d)$ if
\begin{equation*}
    Y_{ij} \iid \mcN(0, (1+\delta_{ij})/d) \textrm{ for } i \leq j.
\end{equation*}
The following celebrated result dates back to \cite{wigner1955characteristic}, and is one of the foundational results of random matrix theory. 
For a modern proof and further results, see~\cite{anderson2010introduction}.
\begin{theorem}[Semicircle law]\label{thm:wigner}
    Let $\bY \sim \GOE(d)$, with eigenvalues $(\lambda_i)_{i=1}^d$, and denote $\mu_\bY \coloneqq (1/d) \sum_{i=1}^d \delta_{\lambda_i}$ its empirical eigenvalue distribution.
    Then, almost surely,
    \begin{equation*}
        \begin{dcases}
            \mu_\bY &\xrightarrow[d \to \infty]{\textrm{weakly}} \rho_{\sci}(\rd x) = \frac{\sqrt{4-x^2}}{2\pi} \indi\{|x| \leq 2\}, \\
            \|\bY\|_\op = \max_{i \in [d]} |\lambda_i| &\xrightarrow[d \to \infty]{} 2.
        \end{dcases}
    \end{equation*}
\end{theorem}
\noindent
$\rho_{\sci}$ in Theorem~\ref{thm:wigner} is often called the \emph{semicircle law}.
Finally, we generically denote constants as $C > 0$ (or $C_1 > 0, C_2 > 0, \cdots$), whose value may vary from line to line.
We will specify their possible dependency on parameters of the problem when relevant.

\subsubsection{Setting of the problem} 

Let $n, d \geq 1$, and $\bW_1, \cdots, \bW_n \iid \GOE(d)$.
For any $\kappa > 0$, we define: 
\begin{equation}
    \label{eq:def_Zkappa}
    Z_\kappa \coloneqq \# \left\{\eps \in \{\pm 1\}^n \, \textrm{s.t.} \, \left\|\sum_{i=1}^n \eps_i \bW_i\right\|_{\op} \leq \kappa \sqrt{n} \right\}.
\end{equation}
We refer to $n^{-1/2} \left\|\sum_{i=1}^n \eps_i \bW_i\right\|_{\op}$ as the \emph{margin} of $\eps \in \{\pm 1\}^n$. 
$Z_\kappa$ thus counts the number of signings $\eps \in \{\pm 1\}^n$ with margin at most $\kappa$.

\myskip
\textbf{The case $\kappa > 2$ --}
Clearly $(1/\sqrt{n})\sum_{i=1}^n \bW_i \sim \GOE(d)$, so that (with $\eps_i = 1$ for all $i \in [n]$), for any $\kappa > 2$, 
\begin{equation}\label{eq:kappa_gtr_2}
    \bbP\left[Z_\kappa \geq 1\right] \geq \bbP_{\bW \sim \GOE(d)}[\|\bW\|_\op \leq \kappa] \aeq 1 - \smallO_d(1),
\end{equation}
using Theorem~\ref{thm:wigner} in $(\rm a)$.
Notice that this bound is also the one given by a random choice of $\eps_i~\iid~\Unif(\{\pm 1\})$ (independently of $\{\bW_i\}_{i=1}^n$), as again in this case 
$n^{-1/2}\sum_{i} \eps_i \bW_i\sim\GOE(d)$.
In what follows, we therefore focus on the (interesting) regime $\kappa \in (0,2]$, in which a choice of signings \emph{dependent on the $\bW_i$} must be made in order to 
get a solution with margin at most $\kappa$.
Notice that this argument implies that Conjecture~\ref{conj:mspencer} holds straightforwardly for i.i.d.\ $\GOE(d)$ matrices.
For this reason, our interest in this problem stems primarily from its interpretation as a constraint satisfaction problem (e.g.\ as a matrix-analog of the SBP) and its connection to questions in random matrix theory, particularly in the critical regime $n = \Theta(d^2)$.
Nevertheless, we also view our understanding of the discrepancy of i.i.d.\ $\GOE(d)$ matrices as a foundational step toward exploring more structured random matrix models, for which Conjecture~\ref{conj:mspencer} is significantly more complex and presents an exciting avenue for future research.

\subsubsection{Asymptotics of the first moment}

Define, for $\kappa \in (0,2]$:
\begin{equation}
    \label{eq:tau_1st_moment}
    \tau_1(\kappa) \coloneqq
        \frac{1}{\log 2} \left[- \frac{\kappa^4}{128} + \frac{\kappa^2}{8} - \frac{1}{2} \log \frac{\kappa}{2} - \frac{3}{8}\right].
\end{equation}
Notice that $\tau_1(2) = 0$, and $\tau_1(\kappa) \to +\infty$ as $\kappa \downarrow 0$.
Our first main result is the following sharp asymptotics for the expectation of $Z_\kappa$.
\begin{theorem}[Asymptotics of the first moment]\label{thm:first_moment}
    Let $\kappa \in (0,2]$, and $\bW_1, \cdots, \bW_n \iid \GOE(d)$.
    Assume $n/d^2 \to \tau \in [0, \infty)$ as $n, d\to \infty$. 
    Then 
    \begin{equation}
        \label{eq:limit_logEZ}
        \lim_{d \to \infty} \frac{1}{d^2} \log \EE Z_\kappa = (\tau - \tau_1(\kappa)) \log 2,
    \end{equation}
    where $Z_\kappa$ is defined in eq.~\eqref{eq:def_Zkappa}.
    In particular, if $\tau < \tau_1(\kappa)$, then
    \begin{equation*}
        \lim_{d \to \infty} \bbP[Z_\kappa = 0] = \lim_{d \to \infty} \bbP\left[\min_{\eps \in \{\pm 1\}^n} \left\|\sum_{i=1}^n \eps_i \bW_i \right\|_\op > \kappa \sqrt{n}\right] = 1.
    \end{equation*}
\end{theorem}
\noindent
In Section~\ref{sec:1st_moment} we carry out the proof of Theorem~\ref{thm:first_moment}. 
It relies on Proposition~\ref{prop:ldp_Wop}, which establishes the large deviations of the operator norm of a 
$\GOE(d)$ matrix, in the scale $d^2$. This is a consequence of now-classical results on the large deviations of the empirical measure 
in so-called $\beta$-matrix models~\citep{arous1997large,anderson2010introduction}, which yield the large deviation rate function in a variational form. 
Further, using results of logarithmic potential theory~\citep{saff2013logarithmic} and the Tricomi theorem~\citep{tricomi1985integral}, we are able to then solve this variational principle,
and its outcome gives eq.~\eqref{eq:tau_1st_moment}.
As a byproduct, we obtain the limiting spectral measure of a matrix $\bW~\sim~\GOE(d)$ constrained on the event $\|\bW\|_\op \leq \kappa$, see Theorem~\ref{thm:lsd_constrained_GOE}.

\myskip
\textbf{Relation to~\cite{kunisky2023online} --}
Theorem~\ref{thm:first_moment} is a refinement of Theorem~1.13 of \cite{kunisky2023online}, 
which shows that $Z_\kappa = 0$ with high probability if $\kappa \leq \delta \cdot 4^{-\tau}$, for an (unspecified) absolute constant $\delta > 0$.
For instance, the bound of \cite{kunisky2023online} correctly predicts $\tau_1(\kappa) \sim - \log \kappa / \log 4$ for $\kappa \to 0$, but fails to capture that $\tau_1(\kappa) \to 0$ continuously as $\kappa \to 2$. 

\subsubsection{The satisfiability region}

We start by introducing a function $\tau_2(\kappa)$, which will serve as a threshold for the validity of our satisfiability analysis.
Let $H(p) \coloneqq - p \log p - (1-p)\log (1-p)$ denote the ``binary entropy'' function.
\begin{proposition}\label{prop:tau2}
   For any $\eta > 0$, let $\delta_\eta \in (0,1)$ to be the unique solution to:
    \begin{equation}
        \label{eq:def_delta_eta}
        H\left(\frac{1+\delta}{2}\right) = \frac{\eta}{1+\eta} \log 2.  
    \end{equation}
    For any $\eta > 0$, let 
\begin{align}
    \nonumber
        \label{eq:def_ttau}
        \ttau(\eta,\kappa) &\coloneqq \max\left\{(1+\eta) \tau_1(\kappa), 
\frac{1+\delta_\eta^2}{2(1-\delta_\eta^2)^2}
    + \left[\frac{\delta_\eta(1+6\delta_\eta+3\delta_\eta^2+2\delta_\eta^3)}{(1-\delta_\eta^2)^3(1-\delta_\eta)}\right] \kappa \right.\\
    &
    \left.
     \hspace{3cm} + \left[ \frac{2(1+\delta_\eta)^5}{(1-\delta_\eta^2)^4} - \frac{(1+3 \delta_\eta^2)}{4(1-\delta_\eta^2)^3}\right] \kappa^2 
      + \frac{\kappa^4(1+3\delta_\eta^2)}{32(1-\delta_\eta^2)^3}
\right\}.
\end{align}
    We define
    \begin{equation}
        \label{eq:def_tau2}
        \tau_2(\kappa) \coloneqq \min_{u \in [0,\kappa]} \min_{\eta > 0} \ttau(\eta, u).
    \end{equation}
    Then $\kappa \mapsto \tau_2(\kappa)$ is a continuous and non-increasing function of $\kappa$.
\end{proposition}
\noindent
Proposition~\ref{prop:tau2} is elementary, we prove it in Section~\ref{subsec:proof_properties_tau2} for completeness, along 
with a straightforward way to evaluate numerically $\tau_2(\kappa)$, see eq.~\eqref{eq:etastar}.
We are now ready to state our main result on the existence of solutions with a given required margin.
It is based on the second moment method, building upon similar techniques to the ones used in the symmetric binary perceptron~\citep{aubin2019storage}.
\begin{theorem}[Satisfiability region]\label{thm:second_moment}
    Let $\kappa \in (0, 2]$. 
    Let $n,d \geq 1$, such that, as $d \to \infty$, $n/d^2 \to \tau > \tau_2(\kappa)$ defined in eq.~\eqref{eq:def_tau2}. 
    For $\bW_1, \cdots, \bW_n~\iid~\GOE(d)$, we have (recall the definition of $Z_\kappa$ in eq.~\eqref{eq:def_Zkappa}):
    \begin{equation*}
        \lim_{d \to \infty} \bbP[Z_\kappa \geq 1] = \lim_{d \to \infty} \bbP\left[\min_{\eps \in \{\pm 1\}^n} \left\|\sum_{i=1}^n \eps_i \bW_i \right\|_\op \leq \kappa \sqrt{n}\right] = 1.
    \end{equation*} 
\end{theorem}
\noindent
In Section~\ref{sec:2nd_moment}, we prove Theorem~\ref{thm:second_moment}. 
We establish concentration of $Z_\kappa$ in Proposition~\ref{prop:2nd_moment}, 
upper bounding $\EE[Z_\kappa^2] / \EE[Z_\kappa]^2$ by some (large) constant as $d \to \infty$. 
Our proof relies on a discrete analog of Laplace's method, combined with showing concentration results (via a log-Sobolev inequality) for the distribution of correlated Gaussian matrices under spectral norm constraints.
We finally strengthen the result of Proposition~\ref{prop:2nd_moment} thanks to the general techniques on sharp transitions for integer feasibility problems 
developed in \cite{altschuler2023zero}, and deduce Theorem~\ref{thm:second_moment}.
Let us emphasize that having access to two-sided bounds on the first moment asymptotics (which are here sharp, and given in Theorem~\ref{thm:first_moment}) 
is crucial in order to develop our second moment analysis.

\subsubsection{Failure of the second moment method}

\noindent
Finally, under a technical assumption, we show that in average-case matrix discrepancy and
for some values of the parameters $(\tau, \kappa)$,
the number of solutions $Z_\kappa$ can be both large in expectation and have large variance, indicating the failure of the 
second moment method approach.
\begin{theorem}[Failure of the second moment method in part of the phase diagram]
    \label{thm:fail_second_moment}
    Assume that Hypothesis~\ref{hyp:control_Gsecond} holds.
    For $\kappa \in (0,2]$, let 
    \begin{equation}\label{eq:tau_f}
        \tau_{\f}(\kappa) \coloneqq \frac{1}{2} \left(\frac{\kappa^2}{4} - 1\right)^4.
    \end{equation}
    Then for $n, d \to \infty$ with $n/d^2 \to \tau$, if $\tau < \tau_\f(\kappa)$:
    \begin{equation}\label{eq:fail_second_moment}
        \liminf_{d \to \infty} \frac{1}{d^2} \log \frac{\EE[Z_\kappa^2]}{\EE[Z_\kappa]^2} > 0.
    \end{equation}
\end{theorem}
\noindent
Theorem~\ref{thm:fail_second_moment} relies on a local analysis of what is referred to as a second moment potential.
The technical Hypothesis~\ref{hyp:control_Gsecond} pertains to controlling the second derivative of this potential uniformly with respect to $d$.
In Section~\ref{sec:fail_2nd_moment}, within the proof of Theorem~\ref{thm:fail_second_moment},
we explain why this assumption is plausible and outline approaches toward a potential proof, noting significant technical challenges that we defer to future work.
In Section~\ref{subsec:discussion}, we further discuss the implications of Theorem~\ref{thm:fail_second_moment} in relation to our other primary results.

\subsection{Discussion and consequences}\label{subsec:discussion}

In Figure~\ref{fig:phase_diagram}, we plot a sketch of the phase diagram of the problem, as established by Theorems~\ref{thm:first_moment},\ref{thm:second_moment} and \ref{thm:fail_second_moment}. 
\begin{figure}[!t]
    \centering
    \includegraphics[width=1.0\textwidth]{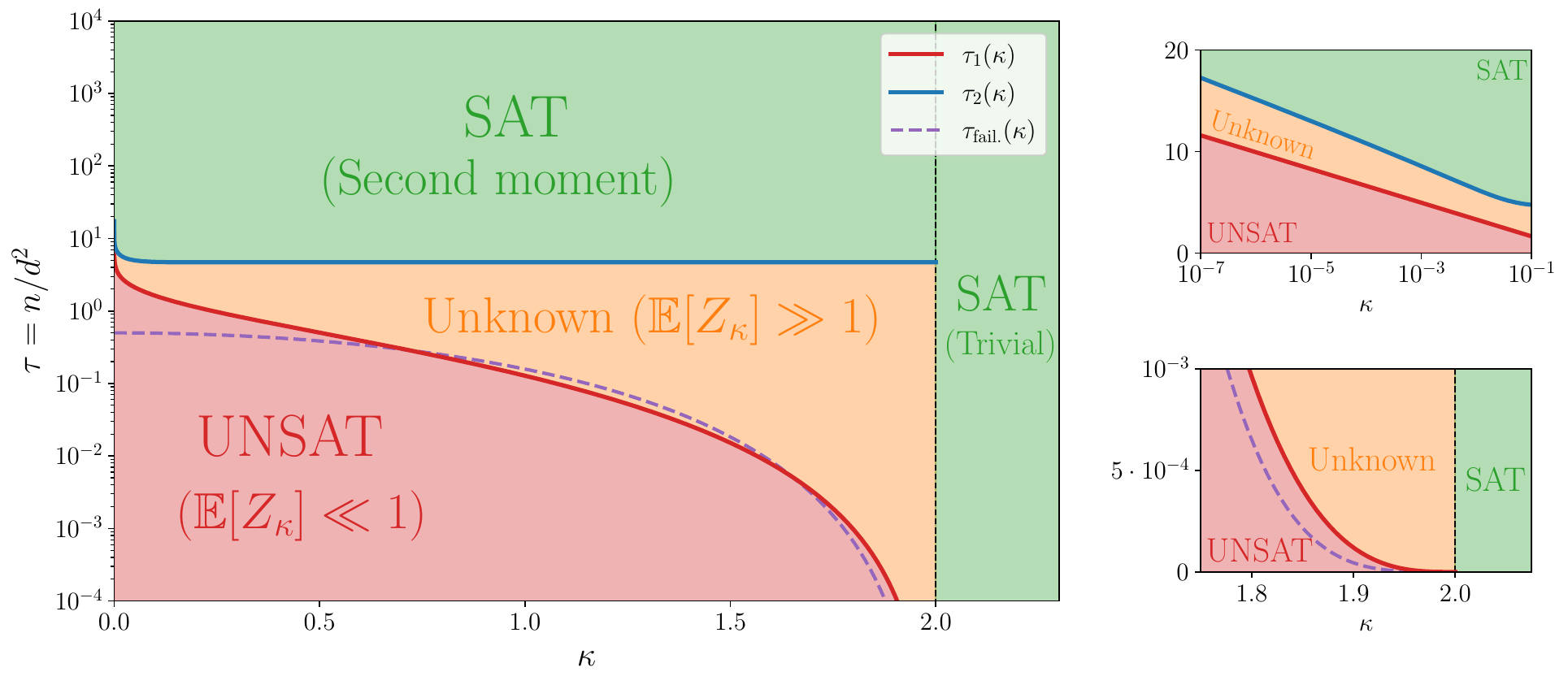}
    \caption{
        Sketch of the satisfiable (SAT) and unsatisfiable (UNSAT) regimes in average-case matrix discrepancy, 
        as proven by Theorems~\ref{thm:first_moment} and \ref{thm:second_moment}. 
        The border of the SAT region for $\kappa < 2$ is given by $\tau_2(\kappa)$, see Proposition~\ref{prop:tau2}. 
        Numerically, we find $\tau_2(\kappa \uparrow 2) \simeq 5.67$.
        For $\kappa >2$, the problem trivially admits a solution, see eq.~\eqref{eq:kappa_gtr_2}.
        The orange region is not characterized by our results, and remains open.
        The dotted purple line shows $\tau_\f(\kappa)$: according to Theorem~\ref{thm:fail_second_moment}, for $\tau_1(\kappa) < \tau < \tau_\f(\kappa)$ the number of solutions 
        has large expectation and variance, and the second moment method fails (see also Fig.~\ref{fig:fail_second_moment}).
        The right plots show the limits $\kappa \downarrow 0$ (top) and $\kappa \uparrow 2$ (bottom).
        We emphasize that $\tau_1(\kappa), \tau_2(\kappa) \to +\infty$ as $\kappa \downarrow 0$.
    \label{fig:phase_diagram}}
\end{figure}
Our results characterize a large part of the $(\kappa, \tau)$ phase diagram: 
we discuss in the following some important consequences, 
and highlight open problems and research directions arising from our analysis.

\myskip
\textbf{Balancing $\Theta(d^2)$ random matrices -- }
An immediate consequence of Theorem~\ref{thm:second_moment} is the following corollary.
\begin{corollary}\label{cor:tau_infty}
    For any $\tau > 0$ large enough\footnotemark%
    , if $n,d \to \infty$ with $n / d^2 \to \tau$, and letting $\bW_1, \cdots, \bW_n \iid \GOE(d)$,
    then (with high probability) there exists $\eps \in \{\pm 1\}^n$ with margin 
    \begin{equation*}
        \frac{1}{\sqrt{n}} \left\|\sum_{i=1}^n \eps_i \bW_i\right\|_\op \leq \kappa_c(\tau) < 2.
    \end{equation*}
    Furthermore, $\kappa_c(\tau) \to 0$ as $\tau \to \infty$.
\end{corollary}
\footnotetext{Numerically, we find $\tau \geq 6$ is enough, see Fig.~\ref{fig:phase_diagram}.}
\noindent
Corollary~\ref{cor:tau_infty} establishes that for $n$ large enough but still in the scale $n = \Theta(d^2)$, one can find solutions 
with margin arbitrarily close to $0$.
As far as we know, our result is the first proof that a number $n = \Theta(d^2)$ of Gaussian random matrices can be balanced in a way to make 
their spectrum \emph{macroscopically shrink} compared to the typical spectrum of a Gaussian matrix: it solves an open problem posed by \cite{kunisky2023online}.

\myskip
\textbf{Tightness of the second moment analysis --}
Similarly to \cite{aubin2019storage} (and subsequent works) for the symmetric binary perceptron (SBP), we use a second moment approach to characterize the feasibility of this problem.  
However, while this method gives a tight threshold for the SBP in Theorem~\ref{thm:trans_SBP} (and thus leaves no ``Unknown'' region in the SBP counterpart to Figure~\ref{fig:phase_diagram}),
the situation in average-case matrix discrepancy is quite different.

\myskip
First, the threshold $\tau_2(\kappa)$ of eq.~\eqref{eq:def_tau2} is likely not optimal, 
as the upper bounds on $\EE[Z_\kappa^2]$ proven in Section~\ref{sec:2nd_moment} are not expected to be tight.
For this reason, the second moment ratio $\EE[Z_\kappa^2]/\EE[Z_\kappa]^2$ might still be bounded by a constant (cf.\ Proposition~\ref{prop:2nd_moment}) even for some values $\tau \leq \tau_2(\kappa)$. 
For instance, one might conjecture that this includes values of $\tau$ arbitrarily close to $0$ when $\kappa$ approaches $2$, which is not captured by our current bounds.
Improving the estimates we obtain in Section~\ref{sec:2nd_moment} to obtain a sharp study of the range of parameters $(\tau, \kappa)$ with bounded second moment ratio is significantly more complex than in the SBP setting:
it requires a precise understanding of the large deviation properties of the law of $(\|\bW_1\|_\op, \|\bW_2\|_\op)$, where $\bW_1, \bW_2$ are two correlated $\GOE(d)$ matrices, both conditioned on having small spectral norm (see eq.~\eqref{eq:def_Gd}). 
A rate function for this law can be obtained in a variational form (see e.g.\ Theorem~3.3 of \cite{guionnet2004first}), however a numerical analysis of this variational formula 
is a very challenging problem, which we leave as a future direction.
\begin{openquestion}[Sharp second moment]\label{op:sharp_2nd_mom}
    Obtain the sharp limit of $(1/d^2) \log \EE[Z_\kappa^2]$, the exponential scale of the second moment.
    This should likely rely on computing numerically $(i)$ the large-$d$ limit of the large deviations rate function of eq.~\eqref{eq:def_Gd}, 
    and $(ii)$ the asymptotic spectral density of two \emph{correlated} $\GOE(d)$ matrices, conditioned to both have spectral norm at most $\kappa$.
\end{openquestion}

\myskip
Furthermore, and very interestingly, we show in Theorem~\ref{thm:fail_second_moment} that 
(assuming Hypothesis~\ref{hyp:control_Gsecond}) the second moment method does not yield a sharp satisfiability threshold in average-case matrix discrepancy. 
Indeed, there exists a range of $\kappa \in (0,2)$ such that 
$\tau_1(\kappa) < \tau_{c}(\kappa)$, so that if $\tau \in (\tau_1(\kappa), \tau_c(\kappa))$, 
then both $\lim (1/d^2) \log \EE[Z_\kappa] > 0$ and $\liminf (1/d^2) \log \EE[Z_\kappa^2]/\EE[Z_\kappa]^2 > 0$.
For such values of $(\tau, \kappa)$, the variance of $Z_\kappa$ is thus exponentially large (in $d^2$), and an approach based solely on the second moment method 
will fail at characterizing the feasibility of the problem.
Numerically, we find  that $\tau_1(\kappa) < \tau_c(\kappa)$ in the range $0.718 \lesssim \kappa \lesssim 1.652$. We summarize these findings in Fig.~\ref{fig:fail_second_moment}.
\begin{figure}[!t]
    \centering
    \includegraphics[width=0.75\textwidth]{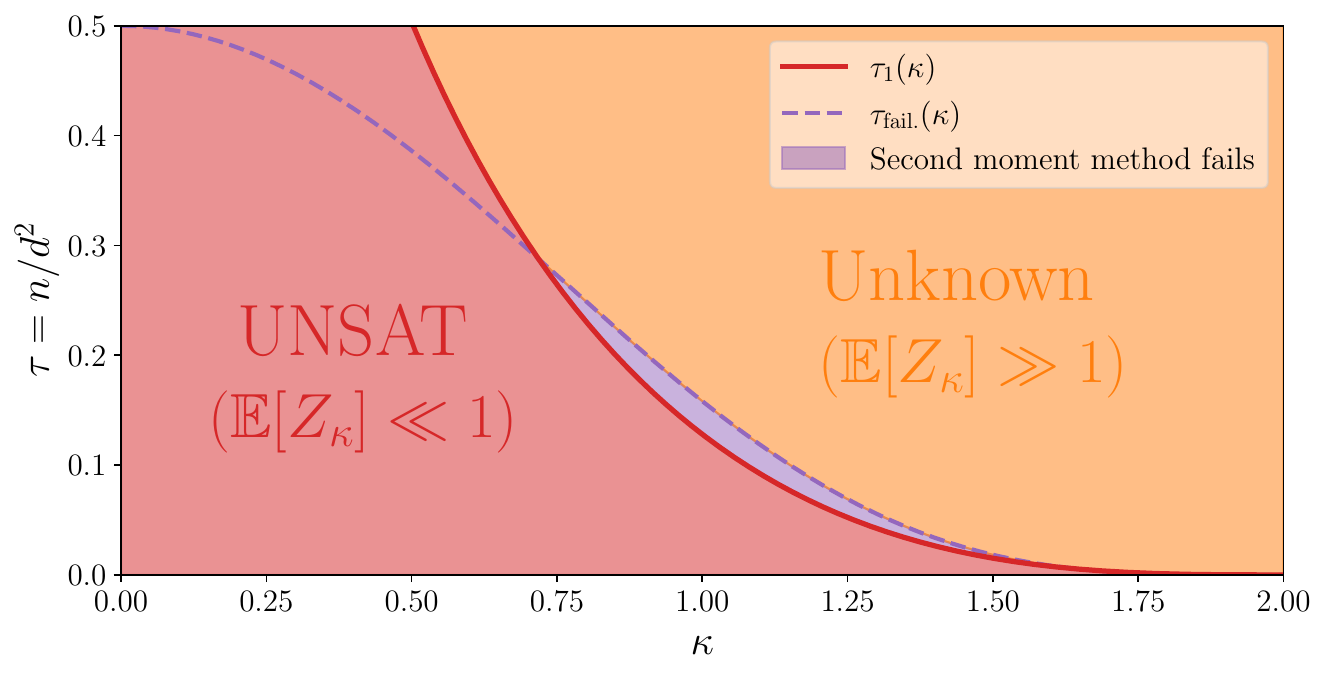}
    \caption{
        Illustration of Theorem~\ref{thm:fail_second_moment}. There exists a region of parameters (in purple) $\tau \in (\tau_1(\kappa), \tau_c(\kappa))$
        for which the number of solutions $Z_\kappa$ to average-case matrix discrepancy satisfies both $\EE[Z_\kappa] \gg 1$ and $\EE[Z_\kappa^2] \gg (\EE[Z_\kappa])^2$,
        so that the second moment method fails at characterizing the feasibility of the problem.
    \label{fig:fail_second_moment}}
\end{figure}

\myskip
We stress that the result of Theorem~\ref{thm:fail_second_moment} is in stark contrast with the symmetric binary perceptron, in 
which the second moment method succeeds in the entirety of the phase diagram (Theorem~\ref{thm:trans_SBP}).
In particular, this unveils a potentially richer picture for the geometry of the solution space than in the SBP.
Finally, we emphasize that the failure of the second moment method might extend to other parts of the ``Unknown'' region of the phase diagram
(beyond the purple region of Fig.~\ref{fig:fail_second_moment}) 
which are not captured by the local analysis used in Theorem~\ref{thm:fail_second_moment}: we refer to Section~\ref{sec:fail_2nd_moment} for details.

\myskip 
\textbf{Sharp threshold sequence --}
Theorem~7 of \cite{altschuler2023zero} (see also Lemma~\ref{lemma:dylan})
directly implies the existence of a sharp threshold sequence for average-case matrix discrepancy. 
Using Theorems~\ref{thm:first_moment} and \ref{thm:second_moment}, we can locate it
in the closure of the ``Unknown region'' of Fig.~\ref{fig:phase_diagram}.
\begin{corollary}[Sharp threshold sequence]\label{cor:sharp_threshold}
    \noindent
    Let $n, d \geq 1$, and assume $n/d^2 \to \tau > 0$ as $d\to \infty$. There exists a sequence 
    $\kappa_c(d,\tau)$ such that: 
    \begin{itemize}
        \item[$(i)$] For any $\eps > 0$, $\bbP[Z_{\kappa_c - \eps} \geq 1] = \smallO_d(1)$, while $\bbP[Z_{\kappa_c + \eps} \geq 1] = 1 - \smallO_d(1)$.
        \item[$(ii)$]  $\tau_1(\kappa_c) \leq \tau \leq \tau_2(\kappa_c)$, i.e.\ $(\kappa_c,\tau)$ is in the closure of the ``Unknown'' region of Fig.~\ref{fig:phase_diagram}.
    \end{itemize}
\end{corollary}
\noindent
The existence of a sharp threshold sequence has been established recently in a variety of other perceptron-like problems, see for instance \cite{talagrand1999self,talagrand2011mean,xu2021sharp,nakajima2023sharp,altschuler2023zero}.
Corollary~\ref{cor:sharp_threshold} shows the existence of a sharp threshold $\kappa_c$ depending on $d$, and bounds it in an interval of size $\Theta(1)$. 
We further conjecture $\kappa_c(d,\tau)$ to converge to a well-defined limit, as we plan to discuss in a future work.
\begin{openquestion}[Sharp SAT/UNSAT transition]\label{op:quenched_limit}
    Obtain the sharp limit of $(1/d^2) \EE \log Z_\kappa$, and from it a sharp SAT/UNSAT threshold, closing the gap in Fig.~\ref{fig:phase_diagram}.
    In particular, compare it with the success of the second moment method (see Open Problem~\ref{op:sharp_2nd_mom}).
\end{openquestion}

\myskip
\textbf{Structure of the solution space --}
Beyond satisfiability, a natural question
concerns
the geometric structure 
of the space of solutions $\{\eps \in \{\pm 1\}^n \, : \, \|\sum_i \eps_i \bW_i\|_\op~\leq~\kappa \sqrt{n}\}$.
In the symmetric binary perceptron, the geometric structure of the solution space was investigated by~\cite{aubin2019storage,perkins2021frozen,abbe2022proof},
and was shown to exhibit a ``frozen-1RSB'' structure: that is, typically, solutions are isolated and far apart.
\begin{openquestion}[Structure of the set of solutions]\label{op:freezing}
    Elucidate whether the solution space (in the SAT region of Fig.~\ref{fig:phase_diagram} for $\kappa < 2$) exhibits the ``frozen-1RSB'' property.
\end{openquestion}
\noindent
In the SBP it was the second moment analysis that suggested the existence of freezing~\citep{aubin2019storage}.
However in the present context two 
difficulties arise: $(i)$ this argument usually relies on establishing  a contiguity property with a ``planted'' version of the model, 
and $(ii)$ even if contiguity holds, the second moment upper bound developed in Section~\ref{sec:2nd_moment} 
is unfortunately not sharp enough to determine whether freezing occurs, 
even withing the SAT region of Fig.~\ref{fig:phase_diagram}.
We discuss points $(i)$ and $(ii)$ in further details in Appendix~\ref{sec:appendix_freezing}.
Going beyond, the failure of the second moment method (discussed above)
may also signal a different geometric structure than in the SBP, at least in parts of the phase diagram.

\myskip 
\textbf{Algorithmic discrepancy --}
The past decade has seen a surge of interest in \emph{algorithmic discrepancy}, that is the design of efficient algorithms to produce signings 
$\eps \in \{\pm 1\}^n$ minimizing a discrepancy objective such as eq.~\eqref{eq:def_disc}. 
In the classical context of vector discrepancy, this line of work was initiated by \cite{bansal2010constructive}, and we refer to \cite{bansal2023resolving,kunisky2023online} for a more detailed description of the 
literature that followed.
For the problem of average-case matrix discrepancy, \cite{kunisky2023online} analyze an online algorithm, introduced originally by~\cite{zouzias2012matrix}, and show that is able to achieve a discrepancy
$\|\sum_i \eps_i \bW_i\|_\op \lesssim d \log(n+d)$ for a large class of random matrices $\bW_i$, including the $\GOE$ distribution.
In the regime $n = \Theta(d^2)$, this bound unfortunately falls short of obtaining a discrepancy lower than $2 \sqrt{n}$. 
\begin{openquestion}[Efficient algorithms]\label{op:algorithms}
    Is there a polynomial-time algorithm which, in the regime $n = \Theta(d^2)$ and with high probability, outputs $\eps \in \{\pm 1\}^n$ such that
   $\|\sum_{i} \eps_i \bW_i\|_\op \leq \kappa \sqrt{n}$
    for some $\kappa < 2$?
\end{openquestion}
\noindent
For the symmetric binary perceptron, which is the vector analog to our problem, it was recently shown that the phase diagram presents a large computational-to-statistical gap
in which low-discrepancy solutions exist, while large classes of polynomial-time algorithms can not find them, and that this was related to the geometric structure of the solution space~\citep{bansal2020line,gamarnik2022algorithms}.

%% file: sections/first_moment.tex
We carry out in this section the proof of Theorem~\ref{thm:first_moment}.
A direct computation from eq.~\eqref{eq:def_Zkappa} using the linearity of expectation yields:
\begin{equation}
    \label{eq:E_Zr}
    \EE Z_\kappa = \sum_{\eps \in \{\pm 1\}^n} \bbP\left[\left\|\sum_{i=1}^n \eps_i \bW_i\right\|_{\op} \leq \kappa \sqrt{n}\right] = 2^n \, \bbP[\|\bW\|_\op \leq \kappa],
\end{equation}
where $\bW \sim \GOE(d)$.
The main result we prove is the following.
\begin{proposition}[Left large deviations for the operator norm of a $\GOE(d)$ matrix]
    \label{prop:ldp_Wop}
    Let $\bW~\sim~\GOE(d)$. For any $\kappa > 0$:
    \begin{equation}
        \label{eq:ldp_Wop}
        \lim_{d \to \infty} \frac{1}{d^2} \log \bbP[\|\bW\|_\op \leq \kappa]
        = 
        \begin{dcases}
            \frac{\kappa^4}{128} - \frac{\kappa^2}{8} + \frac{1}{2} \log \frac{\kappa}{2} + \frac{3}{8} &\textrm{ if } \kappa \leq 2, \\
            0 &\textrm{ otherwise.} 
        \end{dcases}
    \end{equation}
\end{proposition}
\noindent
Using Proposition~\ref{prop:ldp_Wop} in eq.~\eqref{eq:E_Zr} (recall that $n/d^2 \to \tau$) yields eq.~\eqref{eq:limit_logEZ}.
The second result of Theorem~\ref{thm:first_moment} is a direct consequence of Markov's inequality combined with eq.~\eqref{eq:limit_logEZ}.
Notice that Markov's inequality even shows that $\bbP[Z_\kappa > 0]$ goes to zero exponentially fast in $d^2$ for $\tau < \tau_1(\kappa)$.

\myskip 
In the remainder of Section~\ref{sec:1st_moment}, we focus on proving Proposition~\ref{prop:ldp_Wop}.
We note first that for $\kappa > 2$ we have $\log \bbP[\|\bW\|_\op \leq \kappa] = \smallO_d(1)$ as a consequence of Theorem~\ref{thm:wigner}. 
We thus focus on the case $\kappa \in (0,2]$ in what follows.

\myskip 
\textbf{Sketch of proof and important related work --}
    The proof of Proposition~\ref{prop:ldp_Wop} builds upon a large deviation principle (with respect to the weak topology) for the empirical spectral measure of a matrix $\bW$ drawn from $\GOE(d)$, \emph{conditioned} on $\|\bW\|_\op \leq \kappa$.
We study the conditioned law directly, since it is a so-called $\beta$-ensemble (albeit with a singular potential enforcing the spectral norm constraint), 
so that the large deviation analysis follows from the general
results stated in Proposition~\ref{prop:ldp_empirical_measure_conditioned}.
As a direct consequence, we obtain in Corollary~\ref{cor:ldp_Wop_var_principle} the asymptotics of eq.~\eqref{eq:ldp_Wop} as a variational principle over probability measures supported in $[-\kappa,\kappa]$.
We note that one could rely solely on the original large deviations principles for (unconditioned) $\GOE(d)$ matrices proven in~\cite{arous1997large},
however this requires additional care because the set of probability measures supported in $[-\kappa,\kappa]$ has empty interior under the weak topology:
for completeness, we provide in Appendix~\ref{sec:appendix_ldp} a proof of Corollary~\ref{cor:ldp_Wop_var_principle}
that appeared in an earlier version of this work, and which uses only the result of~\cite{arous1997large}.
We solve the resulting variational principle using the theory of logarithmic potentials~\citep{saff2013logarithmic} and Tricomi's theorem~\citep{tricomi1985integral}, similarly to the alternative proof of Wigner's semicircle law obtained by \cite{arous1997large} from their large deviations result.
Similar arguments also appeared in the theoretical physics literature~\citep{dean2006large,vivo2007large,dean2008extreme,majumdar2014top}.
We refer to \cite{anderson2010introduction,guionnet2022rare} for more background and open problems in the theory of large deviations for random matrices.

\myskip
\textbf{Remark --}
The following result is a byproduct of our analysis\footnote{
After the present work was made available as a pre-print, the author realized that another proof of Theorem~\ref{thm:lsd_constrained_GOE} had appeared in an unpublished manuscript~\citep{bouali2015constrained}.
}.
\begin{theorem}[Limiting spectral density of a constrained $\GOE(d)$ matrix]\label{thm:lsd_constrained_GOE}
   For $\kappa \in (0,2]$, denote $\bbP_\kappa$ the law of $\bW \sim \GOE(d)$ conditioned on $\|\bW\|_\op \leq \kappa$. 
    If $\bW \sim \bbP_\kappa$, then its empirical spectral density $\mu_\bW$ converges weakly (as $d \to \infty$, and a.s.) to $\mu_\kappa(\rd x) \coloneqq \rho_\kappa(x) \rd x$ given, for $x \in (-\kappa, \kappa)$, by:
    \begin{equation}
        \label{def:rhokappa_1}
        \rho_\kappa(x) \coloneqq \frac{4+\kappa^2-2x^2}{4 \pi \sqrt{\kappa^2 - x^2}}. 
    \end{equation}
    Notice that $\rho_{2}(x) = \sqrt{4 - x^2}/(2\pi)$ is Wigner's semicircle law, see Theorem~\ref{thm:wigner}.
\end{theorem}
\noindent
We illustrate the form of $\rho_\kappa(x)$ in Figure~\ref{fig:rho_kappa}.
Theorem~\ref{thm:lsd_constrained_GOE} is proven in Section~\ref{subsec:proof_lsd_constrained_GOE}.
\begin{figure}[!t]
   \centering
\includegraphics[width=0.9\textwidth]{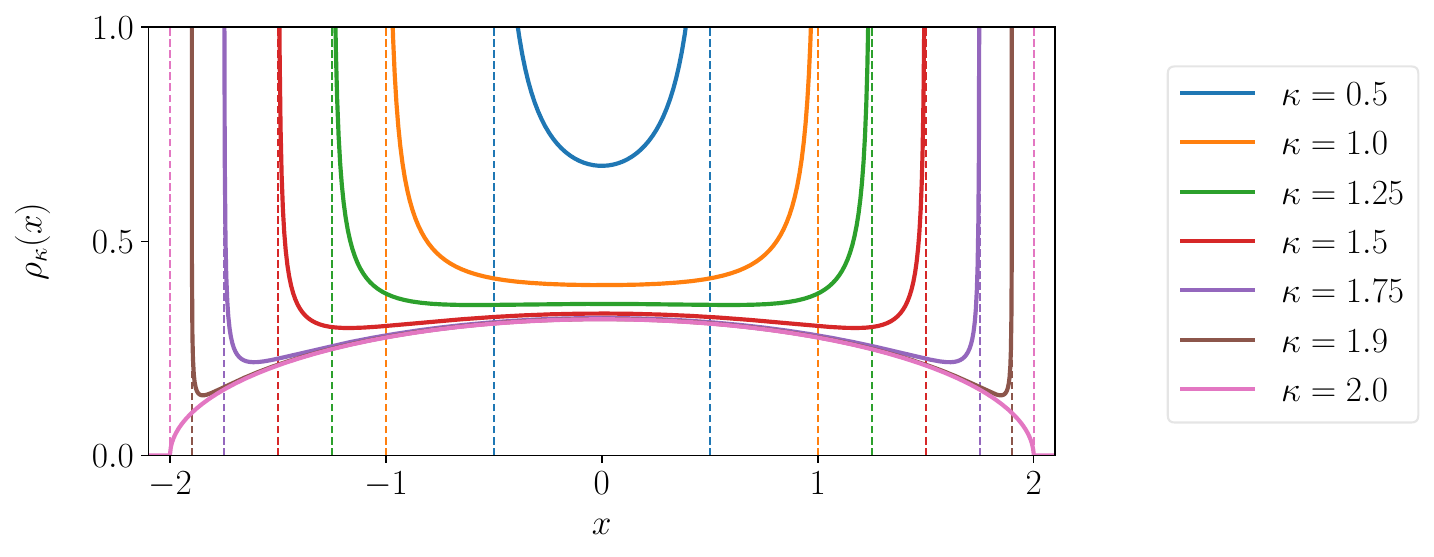}
\caption{$\rho_\kappa(x)$ of eq.~\eqref{def:rhokappa_1} for different values of $\kappa$. For $\kappa = 2$ one recovers the semicircle law.\label{fig:rho_kappa}}
\end{figure}

\myskip
Let us denote
$S_\kappa \coloneqq \{\eps \in \{\pm 1\}^n \, \textrm{s.t.} \, \|\sum_{i=1}^n \eps_i \bW_i\|_{\op} \leq \kappa \sqrt{n} \}$,
so that $Z_\kappa = \# S_\kappa$.
Anticipating our satisfiability results, we note that even if $\bbP[Z_\kappa \geq 1] = 1 - \smallO_d(1)$ as shown in Theorem~\ref{thm:second_moment} for sufficiently large $\tau = n/d^2$, Theorem~\ref{thm:lsd_constrained_GOE} does not ensure that the asymptotic spectrum of $\sum_i \eps_i \bW_i$ (for $\eps \sim \Unif(S_\kappa)$) equals $\rho_\kappa$.
However, one can easily prove that this would be implied e.g.\ by the condition $\EE[Z_\kappa]^2 = (1+\smallO_d(1)) \EE[Z_\kappa]^2$, which is strictly stronger than $\bbP[Z_\kappa \geq 1] = 1 - \smallO_d(1)$.
Unfortunately, the bounds established in Section~\ref{sec:2nd_moment} to prove Theorem~\ref{thm:second_moment} only yield $\EE[Z_\kappa]^2 \leq C \EE[Z_\kappa]^2$ in the satisfiable phase, for a possibly large $C > 1$.
For this reason, we leave the establishment of the limiting spectral density of $\sum_i \eps_i \bW_i$ (for $\eps \sim \Unif(S_\kappa)$) as an open question, 
with $\rho_\kappa$ a natural conjecture.

\subsection{Large deviations results}

Our proof leverages a large deviation principle for the empirical eigenvalue distribution in a large class of matrix models.
Such results were initiated in the mathematics literature by~\cite{arous1997large} in the case of $\GOE(d)$ matrices, and have been then extended, see e.g.~\cite{guionnet2022rare} for a discussion.

\myskip
The space $\mcM_1^+(\bbR)$ of probability measures on $\bbR$ is endowed with its usual weak topology. 
It is metrizable by the Dudley metric~\citep{bogachev2007measure}:
\begin{equation}
    \label{eq:dudley}
    d(\mu, \nu) \coloneqq \left\{\left|\int f \rd \mu - \int f \rd v\right| \, : \, |f(x)| \leq 1 \textrm{ and } |f(x) - f(y)| \leq |x - y|, \, \forall (x, y) \in \bbR^2\right\}.
\end{equation}
Recall that $\Sigma(\mu) \coloneqq \int \mu(\rd x) \mu(\rd y) \log |x-y|$.
The following statement is borrowed from~\cite{fan2015convergence} (variants can be found e.g.\ in~\cite{arous1997large,anderson2010introduction}).
\begin{proposition}[\cite{fan2015convergence}]
    \label{prop:ldp_empirical_measure_conditioned}
    Let $B \subseteq \bbR$ an interval, and $V : B \to \bbR$ a continuous function such that $\lim_{x \to \pm \infty} \frac{V(x)}{2 \log |x|} = +\infty$ (if $B$ is not bounded).
    For $\blambda \in \bbR^d$, let
    \begin{align}\label{eq:coulomb_gas_constraint}
        \mu_{d}^{V,B}(\rd \blambda) \coloneqq \frac{1}{Z_{d}^{V,B}} \prod_{i=1}^d \left(\rd \lambda_i \, \indi\{\lambda_i \in B\} \right) \, e^{-\frac{d}{2} \sum_{i=1}^d V(\lambda_i)} \, \prod_{i < j} |\lambda_i - \lambda_j|,
    \end{align}
    with $Z_d^{V,B}$ a normalization factor, which ensures that $\int \mu_d^{V, B}(\rd \blambda) = 1$.\\
    Denote $\nu_\blambda \coloneqq (1/d) \sum_{i=1}^d \delta_{\lambda_i}$ the empirical measure of $\blambda$.
    Then the law of $\nu_\blambda$ satisfies a large deviation principle, in the scale $d^2$, with good rate function $I_{V,B}(\nu) \coloneqq \mcE_V(\nu) - \inf_{\nu \in \mcM_1^+(B)}[\mcE_V(\nu)]$, where 
    \begin{align}\label{eq:E_V}
        \mcE_V(\nu) \coloneqq - \frac{1}{2} \Sigma(\nu) + \frac{1}{2} \int V(x) \, \nu(\rd x).
    \end{align}
    In particular: 
    \begin{itemize}
        \item[$(i)$] For any $M > 0$, $\{I_{V, B} \leq M\}$ is a compact subset of $\mcM_1^+(B)$.
        \item[$(ii)$]  
        For any open (respectively closed) subset $O \subseteq \mcM_1^+(B)$ (respectively $F \subseteq \mcM_1^+(B)$): 
        \begin{equation*}
            \begin{dcases}
                \liminf_{d\to\infty} \frac{1}{d^2} \log \mu_d^{V,B}[\nu_\blambda \in O] \geq - \inf_{\nu \in O} I_{V,B}(\nu), \\
                \limsup_{d\to\infty} \frac{1}{d^2} \log \mu_d^{V,B}[\nu_\blambda \in F] \leq - \inf_{\nu \in F} I_{V,B}(\nu).
            \end{dcases}
        \end{equation*}
        \item[$(iii)$] 
        \begin{equation*}
            \lim_{d \to \infty} \frac{1}{d^2} \log Z_d^{V,B} = - \inf_{\nu \in \mcM_1^+(B)}[\mcE_V(\nu)].
        \end{equation*}
    \end{itemize}
\end{proposition}
\noindent
 Proposition~\ref{prop:ldp_empirical_measure_conditioned} directly applies to $\bbP_\kappa$, the law of $\bW \sim \GOE(d)$ conditioned on $\|\bW\|_\op \leq \kappa$, 
 with $V(x) = x^2/2$. We obtain from it the following variational formulation for the limit of the left-hand side of eq.~\eqref{eq:ldp_Wop}.
 \begin{corollary}\label{cor:ldp_Wop_var_principle}
    For any $\kappa > 0$, with $\bW \sim \GOE(d)$:
    \begin{equation*}
        \lim_{d \to \infty} \frac{1}{d^2} \log \bbP[\|\bW\|_\op \leq \kappa] = \frac{3}{8} - \inf_{\nu \in \mcM_1^+([-\kappa,\kappa])} \left[-\frac{1}{2} \Sigma(\nu) + \frac{1}{4} \int \nu(\rd x) \, x^2\right].
    \end{equation*}
 \end{corollary}
 \begin{proof}[Proof of Corollary~\ref{cor:ldp_Wop_var_principle}]
    The joint law of the eigenvalues $(\lambda_1, \cdots, \lambda_d)$ of $\bW$ is well-known thanks to the rotation invariance of the law of $\bW$. 
We have (see e.g.\ Theorem~2.5.2 of \cite{anderson2010introduction}):
\begin{align}\label{eq:joint_law_evalues}
    \bbP[\|\bW\|_\op \leq \kappa] &= \frac{\int_{[-\kappa,\kappa]^d} \prod_{i < j} |\lambda_i - \lambda_j| e^{-\frac{d}{4} \sum_{i=1}^d \lambda_i^2} \prod_{i=1}^d \rd \lambda_i}{\int_{\bbR^d} \prod_{i < j} |\lambda_i - \lambda_j| e^{-\frac{d}{4} \sum_{i=1}^d \lambda_i^2} \prod_{i=1}^d \rd \lambda_i}.
\end{align}
The denominator (or partition function) can be computed from Selberg's integrals~\citep{mehta2014random}.
Its limit is given by
\begin{align}\label{eq:part_function_selberg}
    \lim_{d \to \infty} \frac{1}{d^2} \log \int_{\bbR^d} \prod_{i < j} |\lambda_i - \lambda_j| e^{-\frac{d}{4} \sum_{i=1}^d \lambda_i^2} \prod_{i=1}^d \rd \lambda_i 
    &= - \frac{3}{8}.
\end{align}
The result follows then from Proposition~\ref{prop:ldp_empirical_measure_conditioned}-$(iii)$.
 \end{proof}

 \myskip
 \textbf{Alternative proof --}
 For reasons of completeness, and as a first version of this manuscript contained it, we detail in Appendix~\ref{sec:appendix_ldp} a proof of Corollary~\ref{cor:ldp_Wop_var_principle} 
 that does not rely on Proposition~\ref{prop:ldp_empirical_measure_conditioned}, but only on its particular case $B = \bbR$ and $V(x) = x^2/2$, i.e.\ when $\mu_{d}^{V,B}$ is the joint law of the eigenvalues of a $\GOE(d)$ matrix.
 This setting was tackled in the seminal work of~\cite{arous1997large}.
 
 \subsection{Solving the variational principle: proof of Proposition~\ref{prop:ldp_Wop}}\label{subsec:var_principle_1st_mom}

 For $\mu \in \mcM_1^+(\bbR)$, let 
\begin{equation}\label{eq:def_I}
    I(\mu) \coloneqq -\frac{1}{2} \Sigma(\mu) + \frac{1}{4} \int \mu(\rd x) \, x^2 - \frac{3}{8}.
\end{equation}
 We establish here the following equation, for $\kappa \leq 2$: 
 \begin{align}\label{eq:to_prove_var_1st_mom}
    E_\kappa \coloneqq \inf_{\mu \in \mcM_1^+([-\kappa, \kappa])} I(\mu) = 
    - \frac{\kappa^4}{128} + \frac{\kappa^2}{8} - \frac{1}{2} \log \frac{\kappa}{2} - \frac{3}{8}.
 \end{align}
 Combining eq.~\eqref{eq:to_prove_var_1st_mom} with Corollary~\ref{cor:ldp_Wop_var_principle} yields then Proposition~\ref{prop:ldp_Wop}.

 \myskip 
In order to characterize the minimizer of $I(\mu)$ for $\mu \in \mcM_1^+([-\kappa,\kappa])$, we
rely on classical results of logarithmic potential theory, such as Theorem~1.3 of Chapter~I of \cite{saff2013logarithmic}, see also \cite{mhaskar1985does} and Theorem~2.4 of \cite{arous1997large}.
In our context, the ``admissible weight function'' of~\cite{saff2013logarithmic} reads
\begin{equation*}
    w(x) = \exp\{-x^2/4\} \indi\{|x| \leq \kappa\}.
\end{equation*}
Adapting the results of the aforementioned literature to our notations, we obtain the following theorem.
\begin{theorem}[\cite{saff2013logarithmic}]
    \label{thm:properties_inf_I}
    Let $\kappa > 0$ and
    $E_\kappa \coloneqq \inf_{\mu \in \mcM_1^+([-\kappa,\kappa])} I(\mu)$. Then
    \begin{itemize}
        \item[$(i)$] $E_\kappa < \infty$.
        \item[$(ii)$] There exists a unique $\mu_\kappa^\star \in \mcM_1^+([-\kappa,\kappa])$ such that $I(\mu_\kappa^\star) = E_\kappa$.
        \item[$(iii)$] $\mu_\kappa^\star$ is the unique measure in $\mcM_1^+([-\kappa,\kappa])$ such that for $\mu_\kappa^\star$-almost all $x$: 
        \begin{equation*}
            \int \mu_\kappa^\star(\rd y) \, \log |x - y| = \frac{x^2}{4} + \frac{1}{4} \int \mu_\kappa^\star(\rd y) \, y^2 - \frac{3}{4} - 2 E_\kappa.
        \end{equation*}
    \end{itemize} 
\end{theorem}
\noindent
One can further show that this theorem allows to prove that a candidate measure is the optimal one without computing $E_\kappa$.
\begin{lemma}\label{lemma:log_pot_mukappa}
    For $\kappa \in (0,2]$, assume that $\mu \in \mcM_1^+([-\kappa,\kappa])$ and $C \in \bbR$ are such that for all $x \in (-\kappa,\kappa)$:
    \begin{equation}\label{eq:log_pot_mukappa_hyp}
        \int \mu(\rd y) \, \log |x - y| = \frac{x^2}{4} + C.
    \end{equation}
    Then $\mu = \mu_\kappa^\star$.
\end{lemma}
\noindent
Note that Lemma~\ref{lemma:log_pot_mukappa} can also be seen as a consequence of Theorem~3.3 of Chapter~I of \cite{saff2013logarithmic}. 
We give here a short proof for the sake of completeness.
\begin{proof}[Proof of Lemma~\ref{lemma:log_pot_mukappa} --]
    For any $\sigma, \nu$ real signed measures, we have (recall eq.~\eqref{eq:def_I}) 
    \begin{equation*}
        I(\sigma + \nu) - I(\sigma) = -\frac{1}{2} \Sigma(\nu) + \int \nu(\rd x) \left[\frac{x^2}{4} - \int \sigma(\rd y) \log |x - y|\right]. 
    \end{equation*}
    Applying this formula to $\sigma = \mu$ and $\nu = \mu_\kappa^\star - \mu$, we reach:
    \begin{align*}
        I(\mu_\kappa^\star) - I(\mu) &= -\frac{1}{2} \Sigma(\mu_\kappa^\star - \mu) + \int (\mu_\kappa^\star - \mu)(\rd x) \left[\frac{x^2}{4} - \int \mu(\rd y) \log |x - y|\right], \\ 
         &\aeq -\frac{1}{2} \Sigma(\mu_\kappa^\star - \mu), \\ 
         &\bgeq 0.
    \end{align*}
    We used eq.~\eqref{eq:log_pot_mukappa_hyp} in $(\rm a)$ and the fact that $\mu_\kappa^\star$ has no atom since $I(\mu_\kappa^\star) < \infty$~\citep{arous1997large}, so we can restrict the integral 
    to $x \in (-\kappa,\kappa)$.
    In $(\rm b)$ we used the following classical property of the non-commutative entropy $\Sigma(\mu)$, 
    which can be found e.g.\ as Proposition~II.2.2 in \cite{faraut2014logarithmic}.
    \begin{lemma}\label{lemma:nonc_entropy_pos}
        Let $\mu$ be a signed measure on $\bbR$ with compact support, and such that $\int_\bbR \mu(\rd x) = 0$. 
        Let $\hat{\mu}(t) \coloneqq \int \mu(\rd x) \, e^{i t x}$ be the characteristic function of $\mu$.
        Then 
        \begin{equation*}
            \Sigma(\mu) \coloneqq \int \mu(\rd x) \mu(\rd y) \log |x - y| = - \int_0^\infty \frac{|\hat{\mu}(t)|^2}{t} \rd t.
        \end{equation*}
        In particular, $\Sigma(\mu) \leq 0$.
    \end{lemma}
    \noindent
    This shows that $I(\mu) \leq I(\mu_\kappa^\star) = \inf_{\mu \in \mcM_1^+([-\kappa,\kappa])} I(\mu)$. By $(ii)$ of Theorem~\ref{thm:properties_inf_I}, $\mu = \mu_\kappa^\star$.
\end{proof}
\noindent
Thanks to Lemma~\ref{lemma:log_pot_mukappa}, we can give an exact formula for $\mu_\kappa^\star$ by exhibiting a candidate measure satisfying eq.~\eqref{eq:log_pot_mukappa_hyp}, which in turns implies the exact formula~\eqref{eq:to_prove_var_1st_mom} for $E_\kappa$.
\begin{proposition}
    \label{prop:mukappa_value}
    Let $\kappa \in (0,2]$.
    Recall that $E_\kappa = \inf_{\mu \in \mcM_1^+([-\kappa,\kappa])} I(\mu)$ is reached in a unique measure $\mu_\kappa^\star$.
    Let $\mu_\kappa(\rd x) \coloneqq \rho_\kappa(x) \rd x$ be given, for $x \in (-\kappa,\kappa)$, by:
    \begin{equation}\label{eq:def_rhokappa}
        \rho_\kappa(x) \coloneqq \frac{4+\kappa^2-2x^2}{4 \pi \sqrt{\kappa^2 - x^2}}. 
    \end{equation}
    Then $\mu_\kappa = \mu_\kappa^\star$, and we have eq.~\eqref{eq:to_prove_var_1st_mom}:
    \begin{equation*}
        E_\kappa = I(\mu_\kappa^\star) = - \frac{\kappa^4}{128} + \frac{\kappa^2}{8} - \frac{1}{2} \log \frac{\kappa}{2} - \frac{3}{8}.
    \end{equation*}
\end{proposition}
\noindent
\begin{proof}[Proof of Proposition~\ref{prop:mukappa_value}] 
    By Lemma~\ref{lemma:log_pot_mukappa}, to show that $\mu_\kappa = \mu_\kappa^\star$ it is enough to show that eq.~\eqref{eq:log_pot_mukappa_hyp} 
    holds for $\mu_\kappa$. 
    Since $U(x) \coloneqq \int \mu_\kappa(\rd y) \log |x-y|$ satisfies (in the sense of distributions)
    \begin{equation*}
        U'(x) = \textrm{P.V.} \int \frac{\rho_\kappa(y)}{x-y} \rd y,
    \end{equation*}
    it is enough to check for any $x \in (-\kappa,\kappa)$:
    \begin{equation}\label{eq:to_show_rhokappa}
        \textrm{P.V.}\int_{-r}^r \frac{\rho_\kappa(y)}{x-y} \rd y = \frac{x}{2}, 
    \end{equation}
    where $\textrm{P.V.}$ refers to the principal value. 
    Notice that
    \begin{equation*}
        \textrm{P.V.}\int_{-r}^r \frac{\rho_\kappa(y)}{x-y} \rd y = \lim_{\eps \downarrow 0} \Re\left\{\int_{-r}^r \frac{\rho_\kappa(y)}{x+i\eps-y} \rd y\right\}.
    \end{equation*}
    We compute $G_\kappa(z) \coloneqq \int_{-r}^r \rho_\kappa(y)/(z-y) \rd y$ for all $z$ such that $\Im(z) > 0$.
    Changing variables to $y = \kappa \cos \theta$, and then to $\zeta = e^{i \theta}$, we have:
    \begin{align*}
        G_\kappa(z) &= \frac{1}{4\pi} \int_0^\pi \frac{4+\kappa^2 -2 \kappa^2 \cos^2 \theta}{z - \kappa \cos \theta} \, \rd \theta, \\ 
        &= \frac{1}{8\pi} \int_{-\pi}^\pi \frac{4+\kappa^2 -2 \kappa^2 \cos^2 \theta}{z - \kappa \cos \theta} \, \rd \theta, \\
        &= \frac{1}{8 i\pi} \oint_{|\zeta| = 1} \frac{\kappa^2 \zeta^4 - 8 \zeta^2 + \kappa^2}{\zeta^2(\kappa\zeta^2 -2 \zeta z + \kappa)} \rd \zeta.
    \end{align*}
    The denominator has three poles: $\{0, \zeta_-, \zeta_+\}$, where $\zeta_{\pm} \coloneqq (z \pm \sqrt{z^2-\kappa^2})/\kappa$ (we choose the branch of the square root such that $\Im[\sqrt{w}] \geq 0$ for all $w \in \bbC$).
    Since $\Im(z) > 0$, it is easy to show that $|\zeta_-| < 1 < |\zeta_+|$.
    We then apply the residue theorem and find:
    \begin{equation}\label{G_rhokappa}
        G_\kappa(z) = \frac{z}{2} + \frac{4+\kappa^2-2z^2}{4\sqrt{z^2-\kappa^2}}.
    \end{equation}
    Taking $\lim_{\eps \to 0} \Re[G_\kappa(x+i \eps)]$ for $|x| < \kappa$ yields eq.~\eqref{eq:to_show_rhokappa}, and thus $\mu_\kappa = \mu_\kappa^\star$. 
    It remains to compute $E_\kappa = I(\mu_\kappa^\star)$.
    One can use the same arguments as above (based on the residue theorem) to show 
    \begin{equation}
        \label{eq:int_mukappa_x2}
        \int \mu_\kappa^\star(\rd x) \, x^2 = \frac{\kappa^2(8-\kappa^2)}{16}. 
    \end{equation}
    By $(iii)$ of Theorem~\ref{thm:properties_inf_I} and Lemma~\ref{lemma:log_pot_mukappa},
    we then have for all $x \in (-\kappa,\kappa)$:
    \begin{equation}
        \label{eq:Ekappa_1}
        E_\kappa = \frac{\kappa^2(8-\kappa^2)}{128} - \frac{3}{8} + \frac{x^2}{8} - \frac{1}{2} \int_{-\kappa}^\kappa \rho_\kappa(y) \, \log |x-y| \, \rd y.  
    \end{equation}
    Notice that eq.~\eqref{G_rhokappa} is valid for all $z$ with $\Im(z) > 0$. In particular, we reach from it, that for all $x \geq 0$:
    \begin{equation*}
        \textrm{P.V.}\int_{-r}^r \frac{\rho_\kappa(y)}{x-y} \rd y =
        \begin{dcases}
            \frac{x}{2} & \textrm{ if } x < \kappa, \\ 
            \frac{x}{2} + \frac{4+\kappa^2-2x^2}{4\sqrt{x^2-\kappa^2}} & \textrm{ if } x > \kappa.
        \end{dcases}
    \end{equation*}
    Since this is an integrable function, we have
    \begin{equation*}
        \int_{-\kappa}^\kappa \rho_\kappa(y) \, \log |x-y| \, \rd y =
        \begin{dcases}
            \frac{x^2}{4} + C & \textrm{ if } x \leq \kappa, \\ 
            \frac{x^2}{4} + C - \frac{x\sqrt{x^2-\kappa^2}}{4} - \log\left(\frac{x-\sqrt{x^2-\kappa^2}}{\kappa}\right) & \textrm{ if } x > \kappa.
        \end{dcases}
    \end{equation*}
    Using that $\int_{-\kappa}^\kappa \rho_\kappa(y) \log |x-y| \rd y - \log x \to 0$ as $x \to \infty$ yields $C = \log(\kappa/2) - \kappa^2/8$.
    Eq.~\eqref{eq:Ekappa_1} becomes:
    \begin{align*}
        E_\kappa = I(\mu_\kappa^\star) &= \frac{\kappa^2(8-\kappa^2)}{128} - \frac{3}{8}  - \frac{1}{2} \left[\log \frac{\kappa}{2} - \frac{\kappa^2}{8}\right], \\ 
        &= - \frac{\kappa^4}{128} + \frac{\kappa^2}{8} - \frac{1}{2} \log \frac{\kappa}{2} - \frac{3}{8},
    \end{align*}
    which ends the proof.
\end{proof}
\myskip
\textbf{Remark: predicting the form of $\rho_\kappa$ --} In order to predict the density $\rho_\kappa$ given by eq.~\eqref{eq:def_rhokappa}, 
we used an argument based on a heuristic application of Tricomi's theorem~\citep{tricomi1985integral}, which states that if eq.~\eqref{eq:to_show_rhokappa} is satisfied 
and $\rho_\kappa$ is supported on $[-\kappa,\kappa]$, then 
\begin{equation}\label{eq:tricomi}
    \rho_\kappa(x) = \frac{1}{\pi \sqrt{\kappa^2-x^2}}\left[C - \frac{1}{\pi} \textrm{P.V.} \int_{-\kappa}^\kappa \frac{\sqrt{\kappa^2-y^2}}{x-y} \times \left(\frac{y}{2}\right) \rd y\right],
\end{equation}
for some constant $C$ chosen to ensure $\int_{-\kappa}^\kappa \rho_\kappa(x) \rd x = 1$.
A careful evaluation of eq.~\eqref{eq:tricomi} based on the residue theorem yields eq.~\eqref{eq:def_rhokappa}.

\subsection{Limiting spectral distribution of a norm-constrained Gaussian matrix}\label{subsec:proof_lsd_constrained_GOE}

Theorem~\ref{thm:lsd_constrained_GOE} follows as a direct consequence of Proposition~\ref{prop:ldp_empirical_measure_conditioned} and Proposition~\ref{prop:mukappa_value}.
Indeed, let $\bbP_\kappa$ be the law of $\bW \sim \GOE(d)$ conditioned on $\|\bW\|_\op \leq \kappa$, and denote $\mu_\bW$ the empirical spectral distribution of $\bW$.
Then for any $\delta > 0$, if $B(\mu_\kappa^\star,\delta) \subseteq \mcM_1^+([-\kappa,\kappa])$ is the open ball of radius $\delta$ centered in $\mu_\kappa^\star$ for the distance of eq.~\eqref{eq:dudley}, 
\begin{align*}
    \limsup_{d \to \infty} \frac{1}{d^2} \log \bbP_\kappa[\mu_\bW \notin B(\mu_\kappa^\star, \delta)] \leq - \inf_{\mu \in B(\mu_\kappa^\star, \delta)^c} [J_\kappa(\mu)],
\end{align*}
with 
\begin{align*}
    J_\kappa(\mu) \coloneqq I(\mu) - I(\mu_\kappa^\star),
\end{align*}
for $I$ given in eq.~\eqref{eq:def_I}.
Since $J_\kappa$ is a good rate function (see Proposition~\ref{prop:ldp_empirical_measure_conditioned}) and has a unique minimizer (cf.\ $(ii)$ of Theorem~\ref{thm:properties_inf_I}), 
$\inf_{\mu \in B(\mu_\kappa^\star, \delta)^c} [J_\kappa(\mu)] > 0$. Therefore, by the Borel-Cantelli lemma,
\begin{align*}
    \bbP[\limsup_{d \to \infty} d(\mu_\bW, \mu_\kappa^\star) \leq \delta] = 1,
\end{align*}
which ends the proof by taking the limit $\delta \to 0$.
$\qed$

%% file: sections/second_moment.tex
Section~\ref{subsec:proof_properties_tau2} is dedicated to studying the properties of the threshold $\tau_2(\kappa)$. 
The proof of Theorem~\ref{thm:second_moment}, which is the main goal of this section, is outlined in Section~\ref{subsec:reduction_2nd_moment_ub},
and details are given in the remainder of Section~\ref{sec:2nd_moment}.

\subsection{Properties of the satisfiability bound}\label{subsec:proof_properties_tau2}

We prove here Proposition~\ref{prop:tau2}, which follows from the following lemma.
\begin{lemma}\label{lemma:tau2_prop}
    Define, for any $\kappa > 0$, $\bartau(\kappa) \coloneqq \min_{\eta > 0} \ttau(\eta, \kappa)$.
    Then $\bartau(\kappa) = \ttau(\eta^\star(\kappa),\kappa)$,
where
$\eta^\star(\kappa)$ is the unique value of $\eta > 0$ such that:
\begin{align}
    \label{eq:etastar}
      (1+\eta) \tau_1(\kappa) &=
\frac{1+\delta_\eta^2}{2(1-\delta_\eta^2)^2}
    + \left[\frac{\delta_\eta(1+6\delta_\eta+3\delta_\eta^2+2\delta_\eta^3)}{(1-\delta_\eta^2)^3(1-\delta_\eta)}\right] \kappa \\
    &
    + \left[ \frac{2(1+\delta_\eta)^5}{(1-\delta_\eta^2)^4} - \frac{(1+3 \delta_\eta^2)}{4(1-\delta_\eta^2)^3}\right] \kappa^2 
      + \frac{\kappa^4(1+3\delta_\eta^2)}{32(1-\delta_\eta^2)^3}.
\end{align}
Moreover, $\kappa \mapsto \bartau(\kappa)$ is a continuous function of $\kappa$.
\end{lemma}
\noindent
Recall that $\tau_2(\kappa) = \min_{u \in [0,\kappa]} \bartau(u)$, as defined in Lemma~\ref{lemma:tau2_prop}, 
so that $\kappa \mapsto \tau_2(\kappa)$ is clearly continuous and non-increasing.
Moreover, solving eq.~\eqref{eq:etastar} gives a simple way to numerically evaluate $\kappa \in [0,2] \mapsto \bartau(\kappa)$, which then yields the values of $\tau_2$.

\begin{proof}[Proof of Lemma~\ref{lemma:tau2_prop}]
Let 
\begin{equation*}
    \begin{dcases}
    f_\kappa(\eta) &\coloneqq (1+\eta) \tau_1(\kappa), \\
    g_\kappa(\eta) &\coloneq \frac{1+\delta_\eta^2}{2(1-\delta_\eta^2)^2}
    + \left[\frac{\delta_\eta(1+6\delta_\eta+3\delta_\eta^2+2\delta_\eta^3)}{(1-\delta_\eta^2)^3(1-\delta_\eta)}\right] \kappa \\
    &
    + \left[ \frac{2(1+\delta_\eta)^5}{(1-\delta_\eta^2)^4} - \frac{(1+3 \delta_\eta^2)}{4(1-\delta_\eta^2)^3}\right] \kappa^2 
      + \frac{\kappa^4(1+3\delta_\eta^2)}{32(1-\delta_\eta^2)^3}.
    \end{dcases}
\end{equation*}
Recall that $\delta_\eta$ is defined as the unique solution to $H[(1+\delta)/2]/\log 2 = \eta / (1+\eta)$, 
with $H(p) = -p \log p - (1-p) \log(1-p)$.
If $G(\delta) \coloneqq H[(1+\delta)/2]/\log 2$, then
$G$ is smooth and strictly decreasing on $[0,1]$. 
So $\delta_\eta = G^{-1}[\eta / (1+\eta)]$ is a smooth and strictly decreasing function of $\eta > 0$.
It is then immediate by elementary arguments (it is easy to see that each of the $\kappa$-coefficients of $g_\kappa(\eta)$ is a strictly increasing function of $\delta$, similarly to what is done e.g.\ in eq.~\eqref{eq:fprime_q}) that 
$g_\kappa(\eta)$ is also a smooth and strictly decreasing function of $\eta$.

\myskip
Moreover, we have $g_\kappa(0^+) = \infty$, $f_\kappa(0^+) = \tau_1(\kappa) < \infty$, and $g_\kappa(\infty) < \infty$, $f_\kappa(\infty) = \infty$.
It is then elementary to show that for any $\kappa > 0$, $\min_{\eta > 0} \ttau(\eta, \kappa) = \min_{\eta > 0} \max\{f_\kappa(\eta), g_\kappa(\eta)\}$ 
is reached in a unique $\eta^\star(\kappa)$, such that $f_\kappa(\eta^\star(\kappa)) = g_\kappa(\eta^\star(\kappa))$, and that 
$\eta^\star(\kappa)$ is a continuous function of $\kappa$, which implies that $\bartau(\kappa) = f_\kappa(\eta^\star(\kappa))$ is also continuous.
\end{proof}

\subsection{Reduction to a second moment upper bound}\label{subsec:reduction_2nd_moment_ub}

The main element of our analysis is the following upper bound.
\begin{proposition}\label{prop:2nd_moment}
    Let $\kappa \in (0, 2]$. Recall the definition of $Z_\kappa$ in eq.~\eqref{eq:def_Zkappa}, and of $\bartau(\kappa)$ in Lemma~\ref{lemma:tau2_prop}.
    Assume that $\tau > \bartau(\kappa)$.
    Then, for $n, d \to \infty$ with $n/d^2 \to \tau$:
    \begin{equation*} 
        \limsup_{d \to \infty} \frac{\EE[Z_\kappa^2]}{\EE[Z_\kappa]^2} \leq L \cdot \left[1 - \frac{\bartau(\kappa)}{\tau}\right]^{-1/2},
    \end{equation*}
    for an absolute constant $L > 0$.
\end{proposition}
\noindent
We will detail the proof of Proposition~\ref{prop:2nd_moment} in Section~\ref{subsec:proof_2nd_moment}, deferring some intermediate results to Section~\ref{subsec:proof_laplace_discrete} and \ref{subsec:lsi}. 
We first show how to deduce Theorem~\ref{thm:second_moment}.

\begin{proof}[Proof of Theorem~\ref{thm:second_moment}]
Since $\tau > \tau_2(\kappa) = \min_{u \in [0,\kappa]} \bartau(u)$ (see Lemma~\ref{lemma:tau2_prop}), 
and $Z_{\kappa'} \leq Z_\kappa$ for any $\kappa' \leq \kappa$, we can assume without loss of generality that $\tau > \bartau(\kappa)$ in order to prove Theorem~\ref{thm:second_moment}.

\myskip
Because the bound on the right-hand side is strictly higher than $1$, Proposition~\ref{prop:2nd_moment} is not strong enough to directly guarantee the existence of solutions with high probability.
This is a recurring challenge in many random constraint satisfaction problems where
$\EE[Z_\kappa^2]/\EE[Z_\kappa]^2 \to C > 1$. 
This occurs e.g.\ 
in the symmetric binary perceptron, the vector analog of our matrix discrepancy task, see \cite{abbe2022proof}, 
and prevents from applying the classical second moment method to get high-probability bounds.
Fortunately, in~\cite{altschuler2023zero} the author develops general techniques on sharp transitions for integer feasibility problems, 
and applies them to show the concentration of the discrepancy $\min_{\eps \in \{\pm 1\}^n} \left\|\sum_{i=1}^n \eps_i \bW_i \right\|_\op$.
\begin{lemma}[Theorem~7 of \cite{altschuler2023zero}]
    \label{lemma:dylan}
    Let $d \geq 1$, and $\bW_1, \cdots, \bW_n \iid \GOE(d)$.
    Let $\disc(\bW_1, \cdots, \bW_n) \coloneqq \min_{\eps \in \{\pm 1\}^n} \left\|\sum_{i=1}^n \eps_i \bW_i \right\|_\op$. 
    Assume that $n/d^2 \to \tau$ as $d \to \infty$. Then there exists $c(\tau) > 0$ such that
    \begin{align*}
        \frac{\EE[\disc(\bW_1, \cdots, \bW_n)]}{\sqrt{\Var[\disc(\bW_1, \cdots, \bW_n)]}} \geq c(\tau) \sqrt{d}.
    \end{align*}
\end{lemma}
\noindent
Let us see how the combination of Lemma~\ref{lemma:dylan} with our second moment estimates (Proposition~\ref{prop:2nd_moment}) 
ends the proof of Theorem~\ref{thm:second_moment}.

\myskip 
    Since $\bartau(\kappa)$ is a continuous function of $\kappa$ by Lemma~\ref{lemma:tau2_prop}, we choose $\delta > 0$ small enough such 
    that $\tau > \bartau(\kappa-\delta)$.
    Let $X \coloneqq \disc(\bW_1, \cdots, \bW_n)$.
    Notice that $X \leq \|\sum_{i=1}^n \bW_i\|_\op$, 
    so that $\EE X \leq 2 \sqrt{n}$ since $(1/\sqrt{n})\sum_{i=1}^n \bW_i \sim \GOE(d)$, 
    and $\EE \|\bY\|_\op \leq 2$ for $\bY \sim \GOE(d)$ (see e.g.\ Exercise~7.3.5 of \cite{vershynin2018high}).
    From Lemma~\ref{lemma:dylan} we thus get 
    \begin{align}\label{eq:varX}
        \Var(X) \leq \frac{4n}{c(\tau)^2 d},
    \end{align}
    where $\Var(X) \coloneqq \EE[(X - \EE X)^2]$.
    Recall that the Paley-Zygmund inequality states that for any random variable $X \geq 0$:
    \begin{align*}
        \bbP[X > 0] \geq \frac{\EE[X]^2}{\EE[X^2]}.
    \end{align*} 
    Applying it to $Z_{\kappa - \delta}$ and using Proposition~\ref{prop:2nd_moment}, we get that 
    \begin{equation}\label{eq:PX_lb}
        \bbP[X \leq (\kappa-\delta) \sqrt{n}] \geq L^{-1} \cdot \left[1 - \frac{\bartau(\kappa-\delta)}{\tau}\right]^{1/2} + \smallO_d(1).
    \end{equation}
    Let us denote $C(\tau,\kappa,\delta) \coloneqq L \cdot [1 - \bartau(\kappa-\delta)/\tau]^{-1/2}$.
    By Chebyshev's inequality and eq.~\eqref{eq:varX}, we further have for all $t > 0$:
    \begin{align*}
        \bbP[X \geq \EE X - t] \geq 1 - \frac{4n}{c(\tau)^2 d t^2}.
    \end{align*}
    In particular, if $t = c(\tau)^{-1} \sqrt{8 C(\tau, \kappa, \delta) n / d}$, we have $\bbP[X \geq \EE X - t] \geq 1 - (2C(\tau,\kappa,\delta))^{-1}$, which combined 
    with eq.~\eqref{eq:PX_lb} implies that  
    \begin{equation*}
        \EE X \leq (\kappa-\delta) \sqrt{n} + c_2(\tau,\kappa,\delta) \sqrt{\frac{n}{d}}, 
    \end{equation*}
    where we redefined the constant $c_2(\tau,\kappa,\delta) > 0$.
    Again by Chebyshev's inequality and eq.~\eqref{eq:varX}, this implies that for all $u > 0$:
    \begin{align*}
        \bbP[X \leq (\kappa-\delta) \sqrt{n} + c_2 \sqrt{\frac{n}{d}} + u] &\geq 1 - \frac{c_1(\tau) n}{d u^2}.
    \end{align*}
    Picking $u = \delta \sqrt{n} - c_2 \sqrt{n/d}$, we have $u \geq (\delta/2) \sqrt{n}$ for $n,d$ large enough, and this yields:
    \begin{align*}
        \bbP[X \leq \kappa \sqrt{n}] &\geq 1 - \frac{4 c_1(\tau)}{\delta^2 d} \to_{d \to \infty} 1,
    \end{align*}
    which ends the proof.
\end{proof}

\subsection{Proof of the second moment upper bound}
\label{subsec:proof_2nd_moment}

We prove here Proposition~\ref{prop:2nd_moment}.
We compute the second moment as: 
\begin{align}\label{eq:2nd_moment_1}
    \nonumber
    \EE[Z_\kappa^2] &= \sum_{\eps, \eps' \in \{\pm 1\}^n} \bbP\left[\left\|\sum_{i=1}^n \eps_i \bW_i\right\|_\op \leq \kappa \sqrt{n} \, \, \textrm{ and } \, \,\left\|\sum_{i=1}^n \eps'_i \bW_i\right\|_\op \leq \kappa \sqrt{n}\right], \\ 
    \nonumber
    &\aeq 2^n \sum_{\eps \in \{\pm 1\}^n} \bbP\left[\left\|\sum_{i=1}^n \bW_i\right\|_\op \leq \kappa \sqrt{n} \, \, \textrm{ and } \, \,\left\|\sum_{i=1}^n \eps_i \bW_i\right\|_\op \leq \kappa \sqrt{n}\right], \\
    &\beq 2^n \sum_{l=0}^n \binom{n}{l} \bbP\left[\|\bW\|_\op \leq \kappa \textrm{ and } \|q_l \bW + \sqrt{1-q_l^2} \bZ\|_\op \leq \kappa\right].
\end{align}
In $(\rm a)$ and $(\rm b)$ we used the rotation invariance of the $\GOE(d)$ distribution. In eq.~\eqref{eq:2nd_moment_1}, 
we changed variables to $l \coloneqq (\langle \eps, \ones_n\rangle + n)/2$ and defined the ``overlap'' $q_l \coloneqq (1/n) \langle \eps, \ones_n \rangle = 2 (l/n) - 1$.
Furthermore, $\bW, \bZ \sim \GOE(d)$ independently. 
We get from eq.~\eqref{eq:2nd_moment_1} that (recall as well eq.~\eqref{eq:E_Zr}):
\begin{align}\label{eq:2nd_moment_2}
    \frac{\EE[Z_\kappa^2]}{\EE[Z_\kappa]^2} &= \frac{1}{2^n} \sum_{l=0}^n \binom{n}{l} \exp\{n G_d(q_l)\}, 
\end{align}
where for $q \in [-1,1]$:
\begin{align}\label{eq:def_Gd}
    G_d(q) \coloneqq&  
    \frac{1}{n} \log \frac{\bbP\left[\|\bW\|_\op \leq \kappa \textrm{ and } \|q \bW + \sqrt{1-q^2} \bZ\|_\op \leq \kappa\right]}{\bbP[\|\bW\|_\op \leq \kappa]^2}.
\end{align}
Recall that $H(p) \coloneqq - p \log p - (1-p)\log (1-p)$.
We will leverage the following lemma, which is based on standard asymptotic techniques, 
and whose proof is deferred to Section~\ref{subsec:proof_laplace_discrete}.
\begin{lemma}\label{lemma:laplace_discrete}
    Let $n \geq 1$, and $F_n : [-1,1] \to \bbR$ such that $F_n(0) = 0$ and $F_n'(0) = 0$.
    Assume that there exists $(\gamma, \delta) > 0$ such that:
    \begin{itemize}
        \item[$(i)$] $\limsup_{n \to \infty} \sup_{|q| \leq \delta} F_n''(q) \leq 1 - \gamma$. 
        \item[$(ii)$] $\limsup_{n \to \infty} \sup_{|q| \geq \delta} \left[F_n(q) + H\left(\frac{1+q}{2}\right) \right] < \log 2$. 
    \end{itemize}
    Then (with $q_l \coloneqq 2l/n - 1 \in [-1,1]$ for $l \in \{0,\cdots, n\}$):
    \begin{align*}
        \limsup_{n \to \infty} \frac{1}{2^n}\sum_{l=0}^n \binom{n}{l} \exp\{n F_n(q_l)\} &\leq \frac{C}{\sqrt{\gamma}},
    \end{align*}
    for a global constant $C > 0$.
\end{lemma}
\noindent
From eq.~\eqref{eq:2nd_moment_2}, in order to finish the proof of Proposition~\ref{prop:2nd_moment}, it suffices to check conditions 
$(i)$ and $(ii)$ of Lemma~\ref{lemma:laplace_discrete} for $G_d$ defined in eq.~\eqref{eq:def_Gd}, for $\tau > \bartau(\kappa)$, $\gamma = 1 - \bartau(\kappa)/\tau$, and some $\delta > 0$.

\myskip
Recall the definition of $\eta^\star(\kappa)$ in Lemma~\ref{lemma:tau2_prop}. 
We let $\delta \coloneqq \delta_{\eta^\star(\kappa)}$ as defined by eq.~\eqref{eq:def_delta_eta}.

\myskip 
\textbf{Condition $(ii)$ --} 
Notice that $G_d(0) = 0$ and that $G_d$ is clearly an even function of $q$, so $G_d'(0) = 0$ (the smoothness of $G_d$ can be shown by direct computation, as we will see in eq.~\eqref{eq:P_W_Z}).
Furthermore, we have the trivial bound:
\begin{align*}
    G_d(q) + H \left(\frac{1+q}{2}\right) &\leq H \left(\frac{1+q}{2}\right)  - \frac{1}{n} \log \bbP[\|\bW\|_\op \leq \kappa].
\end{align*}
Recall that $q \mapsto H[(1+q)/2]$ is even, and strictly decreasing on $[0,1]$.
Using Proposition~\ref{prop:ldp_Wop}, and the definition of $\tau_1(\kappa)$ in eq.~\eqref{eq:tau_1st_moment}, 
we get
\begin{align*}
    \limsup_{d \to \infty} \sup_{|q| \geq \delta} \left[G_d(q) + H \left(\frac{1+q}{2}\right)\right] &\leq H \left(\frac{1+\delta}{2}\right) + \frac{\tau_1(\kappa)}{\tau} \log 2, \\ 
    &\aless H \left(\frac{1+\delta}{2}\right) + \left[\frac{1}{1+\eta^\star(\kappa)}\right] \log 2, \\ 
    &\beq \left[\frac{\eta^\star(\kappa)}{1+\eta^\star(\kappa)} + \frac{1}{1+\eta^\star(\kappa)}\right] \log 2, \\ 
    &= \log 2,
\end{align*}
using $\tau > \bartau(\kappa) = (1+\eta^\star(\kappa))\tau_1(\kappa)$ in $(\rm a)$, 
and the definition of $\delta = \delta_{\eta^\star(\kappa)}$ in $(\rm b)$, cf.\ eq.~\eqref{eq:def_delta_eta}.
We have thus checked condition $(ii)$ of Lemma~\ref{lemma:laplace_discrete}.

\myskip
\textbf{Condition $(i)$ --}
We will show that for any $\delta \in (0,1)$:
\begin{align}\label{eq:to_show_Gd_2nd_derivative}
    \nonumber
    \limsup_{d \to \infty}\sup_{|q| \leq \delta} G_d''(q) \leq 
    \frac{1}{\tau} &\left\{\frac{1+\delta^2}{2(1-\delta^2)^2}
    + \left[\frac{\delta(1+6\delta+3\delta^2+2\delta^3)}{(1-\delta^2)^3(1-\delta)}\right] \kappa \right.\\
    &
    \left.
      + \left[ \frac{2(1+\delta)^5}{(1-\delta^2)^4} - \frac{(1+3 \delta^2)}{4(1-\delta^2)^3}\right] \kappa^2 
      + \frac{\kappa^4(1+3\delta^2)}{32(1-\delta^2)^3}\right\}.
\end{align}

\myskip
Let us first show how eq.~\eqref{eq:to_show_Gd_2nd_derivative} finishes the proof of Proposition~\ref{prop:2nd_moment}. 
We pick $\delta = \delta_{\eta^\star(\kappa)}$. 
By Lemma~\ref{lemma:tau2_prop}, eq.~\eqref{eq:to_show_Gd_2nd_derivative} can be rewritten for this value of $\delta$ 
as 
\begin{align*}
    \limsup_{d \to \infty}\sup_{|q| \leq \delta} G_d''(q) \leq 
    \frac{\bartau(\kappa)}{\tau} = 1 - \gamma,
\end{align*}
with $\gamma \coloneqq (1 - \bartau(\kappa)/\tau)$.
This implies that condition $(i)$ of Lemma~\ref{lemma:laplace_discrete} holds with this value of $\gamma$, 
and thus ends the proof of Proposition~\ref{prop:2nd_moment}, as described above.

\myskip
\textbf{Proof of eq.~\eqref{eq:to_show_Gd_2nd_derivative} --}
There remains to show eq.~\eqref{eq:to_show_Gd_2nd_derivative}.
Let $q \in [0,1)$.
We have ($\rd \bW = \prod_{i \leq j} \rd W_{ij}$ is the Lebesgue measure over the space $\mcS_{d}$ of symmetric matrices):
\begin{align}
    \label{eq:P_W_Z}
    \nonumber
    &\bbP\left[\|\bW\|_\op \leq \kappa \textrm{ and } \|q \bW + \sqrt{1-q^2} \bZ\|_\op \leq \kappa\right] \\ 
    \nonumber
    &= \frac{\int \indi\{\|\bW\|_\op \leq \kappa\} e^{-\frac{d}{4} \Tr[\bW^2]} \bbP\left[\|q \bW + \sqrt{1-q^2} \bZ\|_\op \leq \kappa\right] \rd \bW}{\int  e^{-\frac{d}{4} \Tr[\bW^2]} \rd \bW}, \\
    \nonumber
    &= \frac{\int \indi\{\|\bW\|_\op \leq \kappa\} e^{-\frac{d}{4} \Tr[\bW^2]} \left(\int \rd \bY e^{-\frac{d}{4(1-q^2)}\Tr[(\bY - q \bW)^2]} \indi\{\|\bY\|_\op \leq \kappa\}\right) \rd \bW}{\left(\int  e^{-\frac{d}{4} \Tr[\bW^2]} \rd \bW\right) \left(\int \rd \bY e^{-\frac{d}{4(1-q^2)}\Tr[\bY^2]}\right)}, \\
    &= \frac{\int \indi\{\|\bW\|_\op, \|\bY\|_\op \leq \kappa\} e^{-\frac{d}{4(1-q^2)} (\Tr[\bW^2] + \Tr[\bY^2]) + \frac{dq}{2(1-q^2)} \Tr[\bY \bW]} \rd \bY \rd \bW}{\left(\int \rd \bW e^{-\frac{d}{4}\Tr[\bW^2]}\right)^2 (1-q^2)^{d(d+1)/4}}.
\end{align}
Starting from eq.~\eqref{eq:P_W_Z}, we can compute the derivatives of $G_d(q)$.
We will use the shorthand notation 
\begin{align}
    \label{eq:gibbs_q}
    \langle \cdot \rangle_{q,\kappa} &\coloneqq \frac{\int (\cdot) \indi\{\|\bW\|_\op, \|\bY\|_\op \leq \kappa\} e^{-\frac{d}{4(1-q^2)} (\Tr[\bW^2] + \Tr[\bY^2]) + \frac{dq}{2(1-q^2)} \Tr[\bY \bW]} \rd \bY \rd \bW}{\int \indi\{\|\bW\|_\op, \|\bY\|_\op \leq \kappa\} e^{-\frac{d}{4(1-q^2)} (\Tr[\bW^2] + \Tr[\bY^2]) + \frac{dq}{2(1-q^2)} \Tr[\bY \bW]} \rd \bY \rd \bW},
\end{align}
i.e.\ $\langle \cdot \rangle_{q,\kappa}$ is the law of $(\bW, \bY)$ two correlated $\GOE(d)$ matrices (with correlation $q$), conditioned on the event $\|\bW\|_\op, \|\bY\|_\op \leq \kappa$.
$G_d(q)$ is the log-partition function (or ``free energy'' in statistical physics) of 
this high-dimensional probability measure, 
and taking derivatives will yield averages of observables under $\langle \cdot \rangle_{q,\kappa}$.
We get
\begin{align*}
    G_d'(q) &= \frac{d(d+1) q}{2n (1-q^2)} + \frac{1}{2n} \left\langle -\frac{dq}{(1-q^2)^2} \Tr[\bW^2+\bY^2] 
    + \frac{d(1+q^2)}{(1-q^2)^2} \Tr[\bW\bY]
    \right\rangle_{q,\kappa}.
\end{align*}
Differentiating further, we obtain:
\begin{align}\label{eq:d2G_dq2_1}
    \nonumber
    G_d''(q) &= \underbrace{\frac{d(d+1) (1+q^2)}{2n (1-q^2)^2}}_{\eqqcolon I_1(q)}
    + \underbrace{\frac{1}{2n} \left\langle -\frac{d(1+3q^2)}{(1-q^2)^3} \Tr[\bW^2+\bY^2] 
    + \frac{2dq(3+q^2)}{(1-q^2)^3} \Tr[\bW\bY] \right\rangle_{q,\kappa}}_{\eqqcolon I_2(q)} \\ 
    &+ \underbrace{\frac{1}{4n} \textrm{Var}_{\langle \cdot \rangle_{q,\kappa}} \left(-\frac{dq}{(1-q^2)^2} \Tr[\bW^2+\bY^2] 
    + \frac{d(1+q^2)}{(1-q^2)^2} \Tr[\bW\bY]\right)}_{\eqqcolon I_3(q)}.
\end{align}
We bound successively the different terms $\{I_a\}_{a=1}^3$ in eq.~\eqref{eq:d2G_dq2_1}. 
Since $n/d^2 \to \tau$, we have:
\begin{align}\label{eq:ub_I1}
   \limsup_{d \to \infty}\sup_{|q| \leq \delta} I_1(q) = \frac{1}{\tau} \sup_{|q| \leq \delta }\frac{(1+q^2)}{2(1-q^2)^2} 
    = \frac{(1+\delta^2)}{2 \tau(1-\delta^2)^2}.
\end{align}
Recall that for a real random variable $X$, we define the sub-Gaussian norm $\|X\|_{\psi_2}$ of $X$ as~\citep{vershynin2018high}:
\begin{align*}
    \|X\|_{\psi_2} \coloneqq \inf \{t > 0 \, : \, \EE[\exp(X^2/t^2)] \leq 2\}. 
\end{align*}
To bound $I_2$ and $I_3$, we rely on the following crucial result, which we prove in Section~\ref{subsec:lsi}.
\begin{lemma}[Concentration of moments under $\langle \cdot \rangle_{q,\kappa}$]
    \label{lemma:conc_moments_Pqkappa}
    Let $q \in (-1,1)$, $\kappa \in (0,2]$, and 
    \begin{equation*}
    P(X_1, X_2) \coloneqq \sum_{p\geq 0}\sum_{i_1, \cdots, i_p \in \{1,2\}} a_{i_1 \cdots i_p} X_{i_1} \cdots X_{i_p}
    \end{equation*}
    be a polynomial 
    in two non-commutative random variables $(X_1, X_2)$.
    Let $(\bW, \bY)~\sim~\langle \cdot \rangle_{q,\kappa}$ given by eq.~\eqref{eq:gibbs_q}.
    Then: 
    \begin{equation*}
    \|\Tr \, P(\bW, \bY) - \langle \Tr \, P(\bW, \bY) \rangle_{q,\kappa} \|_{\psi_2} \leq C \sqrt{1+q} 
    \sum_{p\geq 0} p \cdot \kappa^{p-1}\sum_{i_1, \cdots, i_p \in \{1,2\}} |a_{i_1 \cdots i_p}|,
    \end{equation*}
    where $C > 0$ is an absolute constant. Furthermore, we have the fully explicit bound:
    \begin{equation}\label{eq:var_PWY}
        \Var_{\langle \cdot \rangle_{q,\kappa}} [\Tr \, P(\bW, \bY)] \leq 
         2(1+q)
        \left(\sum_{p\geq 0} p \cdot \kappa^{p-1}\sum_{i_1, \cdots, i_p \in \{1,2\}} |a_{i_1 \cdots i_p}|\right)^2.
    \end{equation}
\end{lemma}
\noindent
Lemma~\ref{lemma:conc_moments_Pqkappa} is a consequence of a log-Sobolev inequality we prove for $\langle \cdot \rangle_{q,\kappa}$.

\myskip 
\textbf{Bounding $I_2$ --}
Note that under the law of $\langle \cdot \rangle_{0,\kappa}$ of eq.~\eqref{eq:gibbs_q} when $q = 0$,
$\bW$ is distributed as a $\GOE(d)$ matrix, conditioned to satisfy $\|\bW\|_\op \leq \kappa$.
By Theorem~\ref{thm:lsd_constrained_GOE}, we know that $\mu_\bW$ weakly converges (a.s.) to $\mu_\kappa^\star$.
Since $\int \mu_\bW(\rd x) x^2 = \int \mu_\bW(\rd x) x^2 \indi\{|x| \leq \kappa\}$, we have by the Portmanteau theorem 
and dominated convergence: 
\begin{align}
    \label{eq:limit_TrW2_q0}
    \nonumber
    \lim_{d \to \infty} \frac{1}{d}\langle\Tr \bW^2\rangle_{0, \kappa} &= \int \mu_\kappa^\star(\rd x) \, x^2 \, \indi\{|x| \leq \kappa\} \\ 
    &= \int \mu_\kappa^\star(\rd x) \, x^2, \\
    &\aeq \frac{\kappa^2 (8-\kappa^2)}{16},
\end{align}
using eq.~\eqref{eq:int_mukappa_x2} in $(\rm a)$.
By symmetry $\bW \to -\bW$, we also trivially have
\begin{align}
    \label{eq:limit_TrWY_q0}
    \langle\Tr \bW \bY\rangle_{0, \kappa} &= 0.
\end{align}
Let us denote $P_q(X, Y) \coloneqq - q (X^2 + Y^2) + (1+q^2) X Y$. 
One easily computes from eq.~\eqref{eq:gibbs_q} that 
for any function $\varphi(\bW, \bY)$: 
\begin{align*}
    \frac{\partial}{\partial q}\langle \varphi(\bW, \bY)\rangle_{q,\kappa} 
    &= \frac{d \left[\langle \varphi(\bW, \bY) \cdot \Tr[P_q(\bW, \bY)] \rangle_{q, \kappa} - \langle \varphi(\bW, \bY) \rangle_{q, \kappa} \langle \Tr[P_q(\bW, \bY)] \rangle_{q, \kappa}\right]}{2(1-q^2)^2}, \\ 
    &= \frac{d \left[\langle \left(\varphi(\bW, \bY) - \langle \varphi \rangle_{q, \kappa}\right) \cdot \left(\Tr[P_q(\bW, \bY)] - \langle \Tr[P_q]\rangle_{q, \kappa}\right) \rangle_{q, \kappa}\right]}{2(1-q^2)^2}.
\end{align*}
In particular:
\begin{align}\label{eq:derivative_average_q}
    \left|\frac{\partial}{\partial q}\langle \varphi(\bW, \bY)\rangle_{q,\kappa}\right| 
    &\leq \frac{d \left[\Var_{\langle \cdot \rangle_{q, \kappa}}[\varphi(\bW, \bY)] \cdot \Var_{\langle \cdot \rangle_{q, \kappa}}[\Tr P_q(\bW, \bY)]\right]^{1/2}}{2(1-q^2)^2}.
\end{align}
Using eq.~\eqref{eq:derivative_average_q} and Lemma~\ref{lemma:conc_moments_Pqkappa}, 
we reach that for both $\varphi = \Tr[\bW^2]$ and $\varphi = \Tr[\bW \bY]$: 
\begin{align}\label{eq:bound_derivative_applications}
    \left|\frac{\partial}{\partial q} \langle \varphi(\bW, \bY)\rangle_{q,\kappa}\right| 
    &\leq \frac{2 d \kappa (1+q)^2}{(1-q^2)^2} = \frac{2 d \kappa}{(1-q)^2}.
\end{align}
Integrating eq.~\eqref{eq:bound_derivative_applications}, and combining it with eqs.~\eqref{eq:limit_TrW2_q0} and eq.~\eqref{eq:limit_TrWY_q0}, we get:
\begin{align}\label{eq:bound_TrW2_TrWY_q0}
    \begin{dcases}
        \left|\frac{1}{d}\langle\Tr \bW^2\rangle_{q, \kappa} - \frac{\kappa^2 (8-\kappa^2)}{16} \right| &\leq \frac{2\kappa |q|}{1-q}  + \smallO_d(1), \\
        \left|\frac{1}{d}\langle\Tr \bW \bY\rangle_{q, \kappa} \right| &\leq \frac{2\kappa |q|}{1-q}  + \smallO_d(1),
    \end{dcases}
\end{align}
where $\smallO_d(1)$ is uniform in $q$.
We get from eq.~\eqref{eq:bound_TrW2_TrWY_q0}: 
\begin{align}
    \label{eq:I2_1}
    \nonumber
    \limsup_{d \to \infty}\sup_{|q| \leq \delta} I_2(q) &\leq 
    \frac{1}{2\tau} \max_{|q| \leq \delta }\left[
        \frac{4 \kappa q^2(3+q^2)}{(1-q^2)^3(1-q)} - \frac{1+3q^2}{(1-q^2)^3} \left(\frac{\kappa^2(8 - \kappa^2)}{16} - \frac{2\kappa |q|}{1-q}\right)
    \right], \\ 
    &= 
    \frac{1}{2\tau} \max_{q \in [0, \delta]}\left[
        \frac{2 \kappa q(1+6q+3q^2+2q^3)}{(1-q^2)^3(1-q)} - \frac{\kappa^2 (8-\kappa^2)(1+3q^2)}{16 (1-q^2)^3}
    \right].
\end{align}
If 
\begin{equation*}
f_\kappa(q) \coloneqq \frac{2 \kappa q(1+6q+3q^2+2q^3)}{1-q} -  \frac{\kappa^2 (8-\kappa^2)(1+3q^2)}{16},
\end{equation*}
then for all $q \in [0,1]$ and $\kappa \in [0,2]$:
\begin{align}\label{eq:fprime_q}
    \nonumber
    f_\kappa'(q) &= \frac{\kappa}{8(1-q)^2} \left[16 + 3(64-8\kappa+\kappa^3)q + 6(8+8\kappa-\kappa^3)q^2 + (32-3\kappa(8-\kappa^2))q^4 - 96 q^5\right], \\ 
    \nonumber
    &\geq \frac{\kappa}{8(1-q)^2} \left[16 + 3(64-16)q + 6(8-8)q^2 + (32-3\cdot 2 \cdot(8))q^4 - 96 q^5\right], \\ 
    \nonumber
    &\geq \frac{\kappa}{8(1-q)^2} \left[16+144q-16q^3-96q^5\right], \\ 
    &\ageq \frac{\kappa}{8(1-q)^2} \left[16+32q\right] > 0,
\end{align}
using $q^k \leq q$ for any $k \geq 1$ in $(\rm a)$.
This implies that in eq.~\eqref{eq:I2_1}, the maximum is attained at $q = \delta$, and we get:
\begin{align}\label{eq:ub_I2}
    \limsup_{d \to \infty} \sup_{|q| \leq \delta} I_2(q) 
    \leq
    \frac{1}{2\tau} \left[
        \frac{2 \kappa \delta(1+6\delta+3\delta^2+2\delta^3)}{(1-\delta^2)^3(1-\delta)} - \frac{\kappa^2 (8-\kappa^2)(1+3\delta^2)}{16 (1-\delta^2)^3}
    \right].
\end{align}

\myskip 
\textbf{Bounding $I_3$ --}
We apply eq.~\eqref{eq:var_PWY} of Lemma~\ref{lemma:conc_moments_Pqkappa} to $P_q(X, Y) = - q (X^2+Y^2) + (1+q^2) XY$, which yields:
\begin{align*}
    I_3(q) &\leq \frac{d^2}{4 n (1-q^2)^4} \cdot 2(1+q) \left(2 \kappa [2 q + 1+q^2]\right)^2, \\ 
    &= \frac{2d^2 (1+q)^5 \kappa^2}{n (1-q^2)^4}.
\end{align*}
So finally we get:
\begin{align}\label{eq:ub_I3}
    \limsup_{d \to \infty} \sup_{|q| \leq \delta} I_3(q) &\leq \frac{2(1+\delta)^5 \kappa^2}{\tau (1-\delta^2)^4}.
\end{align}
Combining eqs.~\eqref{eq:ub_I1},\eqref{eq:ub_I2},\eqref{eq:ub_I3} finishes the proof of eq.~\eqref{eq:to_show_Gd_2nd_derivative}. 
As we discussed above, this ends the proof of Proposition~\ref{prop:2nd_moment}.
$\qed$

\subsection{Discrete Laplace's method for a dimension-dependent exponent}\label{subsec:proof_laplace_discrete}

We prove here Lemma~\ref{lemma:laplace_discrete}.
By hypothesis $(ii)$, we fix $\eps > 0$ such that, for $n$ large enough:
\begin{equation}
    \label{eq:bound_Fn_plus_H_eps}
    \sup_{|q| \geq \delta} \left[F_n(q) + H\left(\frac{1+q}{2}\right) \right] \leq \log 2 - \eps.
\end{equation}
Recall the classical inequality:
\begin{align}
    \label{eq:bounds_binomial_crude}
  \binom{n}{l} \leq e^{n H(l/n)}, \hspace{0.5cm} &\textrm{ for } l \in \{0,\cdots, n\}.
\end{align}
Combining eqs.~\eqref{eq:bound_Fn_plus_H_eps} and \eqref{eq:bounds_binomial_crude}, we have 
\begin{align}
    \nonumber
    \frac{1}{2^n}\sum_{l=0}^n \indi\left\{\left|l - \frac{n}{2}\right| > \frac{n \delta}{2}\right\} \binom{n}{l} \exp\{n F_n(q_l)\}
    &\leq \frac{1}{2^n}\sum_{l=0}^n \indi\left\{\left|l - \frac{n}{2}\right|> \frac{n \delta}{2}\right\} \exp\{n(\log 2  - \eps)\}, \\ 
    \label{eq:ub_2nd_mom_large_q}
    &\leq n \exp\{- n \eps\}.
\end{align}
Let $\sigma \in (0, \gamma)$.
By hypothesis $(i)$, we get that for $n$ large enough 
$F_n''(q)\leq(1-\gamma+\sigma)$ for all $|q| \leq \delta$.
Since $F_n(0) = 0$ and $F_n'(0) = 0$, this implies
$F_n(q)\leq(1-\gamma+\sigma) q^2/2$ for all $|q| \leq \delta$.
Therefore,
\begin{align}
    \label{eq:ub_2nd_mom_small_q_1}
    \frac{1}{2^n}\sum_{l=0}^n \indi\left\{\left|l - \frac{n}{2}\right| \leq \frac{n \delta}{2}\right\} \binom{n}{l} \exp\{n F_n(q_l)\}
    &\leq \frac{1}{2^n} \sum_{l=0}^n \binom{n}{l} e^{\frac{n(1-\gamma+\sigma)}{2} q_l^2}.
\end{align}
Recall that $q_l = 2 (l/n) - 1$.
The right-hand side of eq.~\eqref{eq:ub_2nd_mom_small_q_1} can now be analyzed with standard extensions of Laplace's method. 
We use here the following statement, which is a consequence of the proof of Lemma~2 of \cite{achlioptas2002asymptotic}.
\begin{lemma}[\cite{achlioptas2002asymptotic}]
    \label{lemma:achlioptas}
    There exists $B, C > 0$ such that the following holds.
    Let $G$  a real analytic positive function on $[0,1]$, and define for $\alpha \in [0,1]$: 
    \begin{align*}
        g(\alpha) \coloneqq \frac{G(\alpha)}{\alpha^\alpha (1-\alpha)^{1-\alpha}}. 
    \end{align*} 
    If there exists $\alpha_{\max} \in (0,1)$ a strict global maximum of $g$ in $[0,1]$ such that $g''(\alpha_{\max}) < 0$, then for sufficiently large $n$: 
    \begin{align*}
        B \cdot\frac{g(\alpha_{\max})^{n+1/2}}{\sqrt{-g''(\alpha_{\max})}} \leq \sum_{l=0}^n \binom{n}{l} G(l/n)^n\leq C \cdot \frac{g(\alpha_{\max})^{n+1/2}}{\sqrt{\alpha_{\max}(1-\alpha_{\max})(-g''(\alpha_{\max}))}}.
    \end{align*}
\end{lemma}
\noindent 
\textbf{Remark --} Lemma~\ref{lemma:achlioptas} is stated in \cite{achlioptas2002asymptotic}
as 
\begin{align*}
    C_1 \cdot g(\alpha_{\max})^n \leq \sum_{l=0}^n \binom{n}{l} G(l/n)^n\leq C_2 \cdot g(\alpha_{\max})^{n},
\end{align*}
where the constants $C_1, C_2$ might depend on $\alpha_{\max}$ and $g(\alpha_{\max})$. 
Their proof (see Appendix~A of \cite{achlioptas2002asymptotic}) reveals the dependency of $C_1, C_2$ on $\alpha_{\max}$ and $g''(\alpha_{\max})$, which we make explicit here.

\myskip
We apply Lemma~\ref{lemma:achlioptas} in eq.~\eqref{eq:ub_2nd_mom_small_q_1}, with 
\begin{align*}
    G(x) \coloneqq \frac{1}{2} e^{\frac{(1-\gamma+\sigma) (2 x - 1)^2}{2}}
\end{align*}
Let
\begin{align*}
    g(x) \coloneqq \frac{G(x)}{x^x (1-x)^{1-x}} = \frac{e^{\frac{(1-\gamma+\sigma) (2 x - 1)^2}{2}}}{2 x^x (1-x)^{1-x}}.
\end{align*}
It is clear that $g(x) = g(1-x)$ for all $x \in [0,1]$, and moreover
\begin{align*}
    \frac{1}{2}\frac{\rd}{\rd x} (\log g)(x) &= - (1-\gamma+\sigma)(1-2x) + \arctanh(1-2x).
\end{align*}
Since $\arctanh(u) \geq u$ for all $u \in [0,1)$, we get that for all $x \in (0,1/2]$:
\begin{align*}
    \frac{\rd}{\rd x} (\log g)(x) &\geq 2(\gamma-\sigma) (1-2x).
\end{align*}
Combining this with the symmetry $g(x) = g(1-x)$, we obtain that $g$
has a strict global maximum in $x = 1/2$, and we compute $g(1/2) = 1$.
Moreover, we get by direct computation that $g''(1/2) = - 4(\gamma-\sigma)  < 0$.
All in all, we reach that for $n$ large enough:
\begin{align}
    \label{eq:ub_2nd_mom_small_q_2}
     \frac{1}{2^n} \sum_{l=0}^n \binom{n}{l} e^{-\frac{n(1-\gamma)}{2} q_l^2} \leq \frac{C}{\sqrt{\gamma-\sigma}}.
\end{align}
Combining eqs.~\eqref{eq:ub_2nd_mom_large_q},\eqref{eq:ub_2nd_mom_small_q_1} and \eqref{eq:ub_2nd_mom_small_q_2}, we get:
\begin{align*}
    \limsup_{n \to \infty} \frac{1}{2^n}\sum_{l=0}^n \binom{n}{l} \exp\{n F_n(q_l)\} &\leq \limsup_{n \to \infty} [n e^{-n \eps} + C (\gamma-\sigma)^{-1/2}] = C (\gamma-\sigma)^{-1/2}.
\end{align*}
Letting $\sigma \downarrow 0$ ends the proof of Lemma~\ref{lemma:laplace_discrete}. $\qed$

\subsection{Log-Sobolev inequality for the conditioned law of two correlated \texorpdfstring{$\GOE(d)$}{}}
\label{subsec:lsi}

In this section, we start by reminders on log-Sobolev inequalities, before proving such a property 
for the law of eq.~\eqref{eq:gibbs_q}, and finally proving Lemma~\ref{lemma:conc_moments_Pqkappa}.

\subsubsection{Log-Sobolev inequalities and concentration of measure}\label{subsubsec:lsi}

\begin{definition}\label{def:lsi}
    Let $d \geq 1$.
    A probability measure $\mu \in \mcM_1^+(\bbR^d)$ is said to satisfy the \emph{Logarithmic Sobolev Inequality} (LSI) with constant $c > 0$ if, for any differentiable function 
    $f$ in $L^2(\mu)$, we have 
    \begin{equation}
        \label{eq:lsi}
        \int f^2 \log \frac{f^2}{\int f^2 \, \rd \mu} \rd \mu \leq 2 c \int \|\nabla f\|_2^2 \, \rd \mu.
    \end{equation}
\end{definition}
\noindent
We refer the reader to~\cite{guionnet2009large,anderson2010introduction}
for more on the theory of log-Sobolev inequalities and their applications to concentration results in random matrix theory.
A particularly useful consequence of the LSI is the following. 
\begin{lemma}[Herbst]\label{lemma:herbst}
   Assume that $\mu \in \mcM_1^+(\bbR^d)$ satisfies the LSI with constant $c$. Let $G~:~\bbR^d \to \bbR$ be a Lipschitz function, with Lipschitz constant $\|G\|_L$. 
   Then for all $\lambda \in \bbR$: 
   \begin{align}\label{eq:herbst_MGF}
    \EE_\mu \left[e^{\lambda[G - \EE G]}\right] \leq \exp\left\{\frac{c \|G\|_L^2\lambda^2}{2}\right\}.
   \end{align}
   Therefore, for all 
   $\delta > 0$: 
   \begin{equation*}
    \mu(|G - \EE G| \geq \delta) \leq 2 \exp\left\{-\frac{\delta^2}{2 c \|G\|_L^2}\right\}.
   \end{equation*}
\end{lemma}
\noindent
\textbf{Remark --} Notice that Lemma~\ref{eq:herbst_MGF} implies that 
$\|G - \EE G\|_{\psi_2} \leq C \sqrt{c} \|G\|_L$ for some absolute constant $C > 0$,
and moreover (by Taylor expansion close to $\lambda = 0$) we have $\Var(G) \leq c \|G\|_L^2$.

\myskip
Finally, we will use that a necessary condition for a measure to satisfy the LSI is the so-called \emph{Bakry-Emery (BE)} condition.
\begin{theorem}[Theorem~4.4.17 of \cite{anderson2010introduction}]\label{thm:be_implies_lsi}
    Let $d \geq 1$ and $\Phi : \bbR^d \to \bbR$ a $\mcC^2$ function. 
    Assume that $\Phi$ satisfies the Bakry-Emery condition:
    \begin{equation*}
         \Hess \, \Phi (x) \succeq \frac{1}{c} \Id_d,
    \end{equation*}
    for all $x \in \bbR^d$, for some $c > 0$. 
    Then the measure
    \begin{equation*}
        \mu_\Phi(\rd x) \coloneqq \frac{1}{\mcZ} e^{-\Phi(x)} \rd x
    \end{equation*}
    satisfies the LSI with constant $c$.
\end{theorem}

\subsubsection{\texorpdfstring{A log-Sobolev inequality for the law of eq.~\eqref{eq:gibbs_q}}{}}\label{subsubsec:lsi_Pqkappa}

We show the following lemma.
\begin{lemma}\label{lemma:lsi_Pqkappa}
   For any $q \in (-1,1)$ and $\kappa > 0$,
   the law $\langle \cdot \rangle_{q,\kappa}$ of eq.~\eqref{eq:gibbs_q} satisfies the LSI with constant $2(1+q)/d$.
\end{lemma}

\begin{proof}[Proof of Lemma~\ref{lemma:lsi_Pqkappa}]
    Let $\eps > 0$. We denote $\phi_\eps(x) \coloneqq e^{-x^2/(2\eps)}/\sqrt{2\pi\eps}$.
    We define 
    \begin{equation*}
        V_\eps(x) \coloneqq - \log \left(\int_{-\kappa}^\kappa \rd y \, e^{-y^2/2} \, \phi_\eps(x-y)\right).
    \end{equation*}
    \noindent
    \textbf{Reminders on log-concavity --} 
    A real positive integrable function $p$ is said to be \emph{strongly log-concave} with variance parameter $\sigma^2$ (denoted $\SLC(\sigma^2)$)
    if $p(x) = \phi_{\sigma^2}(x) \cdot e^{\varphi(x)}$, for some concave function $\varphi : \bbR \to [-\infty, \infty)$.
    We refer the reader to \cite{saumard2014log} for properties of log-concave and strongly log-concave functions and probability distributions.
    It is clear that $x \mapsto e^{-x^2/2} \indi\{|x| \leq \kappa\}$ is $\SLC(1)$,
    and that $\phi_\eps$ is $\SLC(\eps)$. 
    By Theorem~3.7 of \cite{saumard2014log}, if $f$ is $\SLC(\sigma_1^2)$ and $g$ is $\SLC(\sigma_2^2)$, 
    their convolution $f \star g$ is $\SLC(\sigma_1^2+\sigma_2^2)$. 
    Therefore, $e^{-V_\eps}$ is $\SLC(1 + \eps)$, which implies (since $V_\eps$ is smooth) 
    that $V_\eps''(x) \geq (1+\eps)^{-1}$ for all $x \in \bbR$. Since $V_\eps$ is even,
    and 
    \begin{equation*}
        V_\eps(0) \geq - \log \int_{-\kappa}^\kappa \rd y \, \phi_\eps(y) \geq 0,
    \end{equation*}
    we get that $V_\eps(x) \geq x^2/[2(1+\eps)]$ for all $x \in \bbR$.
    We define $\mu_\eps$ as:
    \begin{align}\label{eq:def_mueps}
    \mu_\eps(\rd \bY, \rd \bW) &\coloneqq \frac{e^{-\frac{d}{2(1-q^2)} (\Tr V_\eps(\bW) + \Tr V_\eps(\bY)) + \frac{dq}{2(1-q^2)} \Tr[\bY \bW]} \rd \bY \rd \bW}{\int e^{-\frac{d}{2(1-q^2)} (\Tr V_\eps(\bW) + \Tr V_\eps(\bY)) + \frac{dq}{2(1-q^2)} \Tr[\bY \bW]} \rd \bY \rd \bW}.
    \end{align} 
    Recall $\mcS_d$ is the set of symmetric $d \times d$ matrices.
    Since $x \mapsto V_\eps(x) - x^2/(2[1+\eps])$ is convex, by Klein's lemma (cf.\ Lemma~4.4.12 of \cite{anderson2010introduction} or Lemma~6.4 of \cite{guionnet2009large}), the function 
    $\bW \mapsto \Tr V_\eps(\bW) - \Tr[\bW^2]/(2[1+\eps])$ is also convex. 
    Thus, for all $\bW \in \mcS_d$: 
    \begin{align*}
        \Hess \, V_\eps(\bW) \succeq \frac{1}{1+\eps} \Id_{\mcS_d}.
    \end{align*}
    All in all we get for any $\bW$ and $\bY$:
    \begin{align*}
       \Hess \left[\frac{d}{2(1-q^2)} (\Tr V_\eps(\bW) + \Tr V_\eps(\bY)) - \frac{dq}{2(1-q^2)} \Tr[\bY \bW]\right]  
       \succeq \frac{d}{2(1-q^2)} 
       \begin{pmatrix}
         \frac{\Id_{\mcS_d}}{1+\eps} & - q \Id_{\mcS_d} \\ 
         -q \Id_{\mcS_d} & \frac{\Id_{\mcS_d}}{1+\eps}
       \end{pmatrix},
    \end{align*}
    which means
    \begin{align*}
        \lambda_\mathrm{min} \left(\Hess \left[\frac{d}{2(1-q^2)} (\Tr V_\eps(\bW) + \Tr V_\eps(\bY)) - \frac{dq}{2(1-q^2)} \Tr[\bY \bW]\right]\right) 
        &\geq \frac{d(1-q-\eps q)}{2(1+\eps)(1-q^2)}.
    \end{align*}
    Therefore, by Theorem~\ref{thm:be_implies_lsi}, $\mu_\eps$ satisfies the LSI with constant 
    \begin{equation*}
        \frac{2(1+\eps)(1-q^2)}{d(1-q-\eps q)} =  \frac{2(1+q)}{d} + \smallO_{\eps \to 0}(1).
    \end{equation*}
    Finally, notice that
    we have $V_\eps(x) \to_{\eps \to 0} V(x)$ pointwise, with $V(x)$ defined as: 
    \begin{align*}
        V(x) &\coloneqq 
        \begin{dcases}
            \frac{x^2}{2} &\textrm{ if } |x| < \kappa, \\
            \frac{\kappa^2}{2} + \log 2 &\textrm{ if } |x| = \kappa, \\
            +\infty &\textrm{ if } |x| > \kappa.
        \end{dcases}
    \end{align*}
    Since $V_\eps(x) \geq x^2/4$ for $\eps \leq 1/2$, we get by dominated convergence and the Portmanteau theorem 
    that $\mu_\eps \to_{\eps\to 0} \mu_0$ weakly, where $\mu_0$ is defined as in eq.~\eqref{eq:def_mueps}, replacing $V_\eps$ by $V$. 
    Because the set $\{\|\bW\|_\op = \kappa\}$ has Lebesgue measure zero, we further have that
    $\mu_0 = \langle \cdot \rangle_{q,\kappa}$.
    Since $\mu_\eps$ satisfies the LSI with constant $2(1+q)/d + \smallO_{\eps \to 0}(1)$, and weakly converges to $\langle \cdot \rangle_{q,\kappa}$ as $\eps \downarrow 0$, we deduce that $\langle \cdot \rangle_{q,\kappa}$ satisfies the LSI\footnote{
    By taking the limit of eq.~\eqref{eq:lsi} for well-behaved functions $f$ using weak convergence, and extending to all differentiable and square integrable functions by density. See e.g.\ the proof of Theorem~4.4.17 in \cite{anderson2010introduction} for details.}
    with constant $2(1+q)/d$.
\end{proof}

\subsubsection{Proof of Lemma~\ref{lemma:conc_moments_Pqkappa}}\label{subsubsec:proof_conc_moments_Pqkappa}

Let $P(X_1, X_2) = \sum_{p\geq 0}\sum_{i_1, \cdots, i_p \in \{1,2\}} a_{i_1 \cdots i_p} X_{i_1} \cdots X_{i_p}$.
We make use of the following elementary result.
\begin{lemma}[Lemma~6.2 of \cite{guionnet2009large}]\label{lemma:poly_lipschitz_Bop}
    Let $Q$ be a polynomial in two non-commutative variables.  
    Then, for any $\kappa > 0$, the function 
    \begin{equation*}
        (\bW , \bY) \in B_\op(\kappa) \times B_\op(\kappa) \mapsto \Tr[Q(\bW, \bY)]
    \end{equation*}
    is Lipschitz with respect to the Euclidean norm, with Lipschitz norm bounded by $\sqrt{d} C(Q, \kappa)$ for some constant $C(Q, \kappa) > 0$. 
    If $Q$ is a monomial of degree $p$, one can take $C(Q, \kappa) = p \kappa^{p-1}$.
\end{lemma}
\noindent
Notice that $\supp(\langle \cdot \rangle_{q,\kappa}) \subseteq B_\op(\kappa) \times B_\op(\kappa)$.
By Lemma~\ref{lemma:poly_lipschitz_Bop}, $f:(\bW, \bY) \mapsto \Tr P(\bW, \bY)$ is thus Lipschitz on the support of $\langle \cdot \rangle_{q,\kappa}$, 
with Lipschitz constant
\begin{equation*}
    \|f\|_L \leq \sqrt{d} \sum_{p \geq 0} p \cdot \kappa^{p-1} \sum_{i_1, \cdots, i_p \in \{1,2\}} |a_{i_1 \cdots i_p}|.
\end{equation*}
Combining Lemmas~\ref{lemma:lsi_Pqkappa} and \ref{lemma:herbst} finishes the proof of Lemma~\ref{lemma:conc_moments_Pqkappa}.
$\qed$

%% file: sections/fail_2nd_moment.tex
We prove here Theorem~\ref{thm:fail_second_moment}, and discuss as well the technical hypothesis on which our statement relies.
Let $\kappa \in (0, 2]$ and $\tau > 0$. We start again from the second moment computation detailed in Section~\ref{subsec:proof_2nd_moment}, 
and more precisely from eq.~\eqref{eq:2nd_moment_2}, which we recall here:
\begin{align}\label{eq:2nd_moment_recall}
    \frac{\EE[Z_\kappa^2]}{\EE[Z_\kappa]^2} &= \frac{1}{2^n} \sum_{l=0}^n \binom{n}{l} \exp\{n G_d(q_l)\}, 
\end{align}
where for $q \in [-1,1]$:
\begin{align*}
    G_d(q) \coloneqq&  
    \frac{1}{n} \log \frac{\bbP\left[\|\bW\|_\op \leq \kappa \textrm{ and } \|q \bW + \sqrt{1-q^2} \bZ\|_\op \leq \kappa\right]}{\bbP[\|\bW\|_\op \leq \kappa]^2}.
\end{align*}
Using eq.~\eqref{eq:P_W_Z} we further have, for any $q \in (-1,1)$:
\begin{align}\label{eq:Gd_integral}
    G_d(q) &= \frac{1}{n} \log \frac{\int \indi\{\|\bW_1\|_\op, \|\bW_2\|_\op \leq \kappa\} e^{-\frac{d}{4(1-q^2)} (\Tr[\bW_1^2] + \Tr[\bW_2^2]) + \frac{dq}{2(1-q^2)} \Tr[\bW_1 \bW_2]} \rd \bW_1 \rd \bW_2}{\left(\int \rd \bW \, \indi\{\|\bW \|_\op \leq \kappa\} \, e^{-\frac{d}{4}\Tr[\bW^2]}\right)^2 (1-q^2)^{d(d+1)/4}}.
\end{align}
We can compute explicitly the limit of $G_d''(q)$ for $q$ close to $0$ as follows.
\begin{lemma}\label{lemma:limit_Gsecond}
    Recall the definition of $\tau_\f(\kappa)$ in eq.~\eqref{eq:tau_f}. We have
    \begin{equation*}
        \lim_{d \to \infty} G_d''(0) = \frac{\tau_\f(\kappa)}{\tau}.
    \end{equation*}
\end{lemma}
\noindent
Lemma~\ref{lemma:limit_Gsecond} relies on the limiting spectral distribution theorem we established in Theorem~\ref{thm:lsd_constrained_GOE}, 
and is proven below. First, we establish how Lemma~\ref{lemma:limit_Gsecond}, 
alongside the following technical hypothesis,
ends the proof of Theorem~\ref{thm:fail_second_moment}.
\begin{hypothesis}\label{hyp:control_Gsecond}
    We assume that, for $G_d$ given in eq.~\eqref{eq:Gd_integral}:
    \begin{equation*}
        \lim_{\eps \downarrow 0}\limsup_{d \to \infty} \sup_{|q| \leq \eps} |G_d''(q) - G_d''(0)| = 0.
    \end{equation*}
\end{hypothesis}
\noindent
\textbf{Discussion of Hypothesis~\ref{hyp:control_Gsecond} --}
Hypothesis~\ref{hyp:control_Gsecond} states that $G_d''(q)$ is continuous in $q = 0$, uniformly in $d$ as $d \to \infty$.
We make two important remarks related to this hypothesis:
\begin{itemize}[leftmargin=*]
    \item First, note that $G_d$ in eq.~\eqref{eq:Gd_integral} can be interpreted as a large deviations rate function for the spectral norms of two $q$-correlated $\GOE(d)$ matrices $(\|\bW_1\|_\op, \|\bW_2\|_\op)$, 
    on the scale $d^2$. 
    In general, the large deviations rate function (in the weak topology) for the joint law of the spectral measures of two matrices $\bW_1$ and $\bW_2$
    drawn from a $\beta$-ensemble with an interacting potential proportional to $\Tr[\bW_1 \bW_2]$ has been established in~\cite{guionnet2004first}[Theorem~3.3]. 
    From these results, one can obtain the existence of the limit of $G_d(q)$ as $d \to \infty$ (which we call $G(q)$), as well as a variational formula for it.
    Under this framework, Hypothesis~\ref{hyp:control_Gsecond} would follow if $G_d''$ was shown to converge uniformly to $G''$ in a neighborhood of $q = 0$.
    \item A sufficient condition for Hypothesis~\ref{hyp:control_Gsecond} to hold is to establish that $\sup_{|q| \leq \eps} |G_d^{(3)}(q)| \leq C(\eps)$ as $d \to \infty$.
    The third derivative of $G_d$ can be explicitly calculated from eq.~\eqref{eq:Gd_integral}, similar to our computation of the second derivative in Section~\ref{sec:2nd_moment}.
    As discussed there, bounding the second derivative required demonstrating that $\Var_{\langle \cdot \rangle}[\Tr P[\bW_1, \bW_2]] = \mcO(1)$, where $P$ is a polynomial independent of $d$, and $\bW_1, \bW_2$ are two correlated $\GOE(d)$ matrices conditioned to have spectral norm at most $\kappa$ (see eq.~\eqref{eq:gibbs_q}).
    We established this bound in Lemma~\ref{lemma:conc_moments_Pqkappa} using classical concentration techniques.
    Note that $\langle \Tr[P(\bW_1, \bW_2)]\rangle = \Theta(d)$, so we established strong concentration properties for this quantity.
    For the third derivative, however, we now essentially need to show that
    \begin{equation}\label{eq:required_3rd_moment}
        \left\langle\left(\Tr P[\bW_1, \bW_2] - \langle \Tr P[\bW_1, \bW_2] \rangle\right)^3\right\rangle = \mcO\left(\frac{1}{d}\right).
    \end{equation}
    Unfortunately, the bound established in Lemma~\ref{lemma:conc_moments_Pqkappa} only yields eq.~\eqref{eq:required_3rd_moment} 
    with a right-hand side of $\mcO(1)$.
    Achieving a sharper bound would require an even more precise control over the statistics of
    $\Tr P[\bW_1, \bW_2]$
    under the law $\langle \cdot \rangle$ of eq.~\eqref{eq:gibbs_q}, 
    which we have not been able to achieve at present and leave as a direction for future work\footnote{
     A bound $\smallO_d(1)$ for the quantity of eq.~\eqref{eq:required_3rd_moment} can be shown if the joint density of $(\bW_1, \bW_2)$ is proportional to $\exp\{-d \Tr V[\bW_1, \bW_2]\}$, for some strongly convex and polynomial potentials $V$, as a consequence 
     of a central limit theorem for $\Tr P[\bW_1,\bW_2]$, see~\cite{guionnet2009large}[Chapter 9]. However even achieving this bound here appears non-trivial because our potential is singular as a consequence of the constraint $\indi\{\|\bW_a\|_\op \leq \kappa\}$ (see e.g.~\cite{borot2013asymptotic} for the case of single-matrix models).
    }.
\end{itemize}

\myskip
We come back to the proof of Theorem~\ref{thm:fail_second_moment}.
If we assume that $\tau < \tau_\f(\kappa)$, by  Lemma~\ref{lemma:limit_Gsecond} and Hypothesis~\ref{hyp:control_Gsecond}, there exists 
$\delta > 0$ and $\eps > 0$ (depending on $\tau, \kappa$) such that, for $d$ large enough: 
\begin{equation}\label{eq:lb_Gd2}
    \inf_{|q| \leq \eps} G_d''(q) \geq (1+\delta).
\end{equation}
Recall that $H(p) \coloneqq - p \log p - (1-p)\log (1-p)$ is the ``binary entropy'' function. We define $S(q) \coloneqq H[(1+q)/2] - \log 2$, 
and
\begin{align}\label{eq:def_Phid}
    \Phi_d(q) \coloneqq G_d(q) + S(q).
\end{align}
Since $S$ is a smooth function of $q$, and $S''(0) = -1$, from eq.~\eqref{eq:lb_Gd2} there exists new constants $(\eps, \delta) > 0$ 
such that
\begin{equation}\label{eq:lb_Phid2}
    \inf_{|q| \leq \eps} \Phi_d''(q) \geq \delta.
\end{equation}
Notice that $\Phi_d(0) = 0$ and $\Phi_d'(0) = 0$ (since $\Phi_d$ is an even function of $q$), so eq.~\eqref{eq:lb_Phid2} implies 
that for $d$ large enough: 
\begin{equation}\label{eq:lb_Phi}
     \inf_{|q| \leq \eps}\left[\Phi_d(q) - \frac{\delta q^2}{2}\right] \geq 0.
\end{equation}
Using the classical inequality that for any $l \in \{0, \cdots, n\}$:
\begin{equation*}
   \binom{n}{l} \geq \frac{1}{n+1} 2^{n H(l/n)},
\end{equation*}
we obtain from eq.~\eqref{eq:2nd_moment_recall}:
\begin{align*}
    \frac{\EE[Z_\kappa^2]}{\EE[Z_\kappa]^2} &\geq \frac{1}{n+1} \sum_{l=0}^n \exp\{n \Phi_d(q_l)\} \ageq \frac{1}{n+1} \sum_{\substack{0 \leq l \leq n \\ |q_l| \leq \eps}} \exp\left\{\frac{n\delta q_l^2}{2}\right\},
\end{align*}
where $q_l = 2 (l/n) - 1$, and we used eq.~\eqref{eq:lb_Phi} in $(\rm a)$.
Choosing $l \in \{0, \cdots, n\}$ such that $\eps/2 \leq |q_l| \leq \eps$, we reach (recall that $n/d^2 \to \tau$):
\begin{align*}
    \liminf_{d \to \infty} \frac{1}{d^2} \log \frac{\EE[Z_\kappa^2]}{\EE[Z_\kappa]^2} &\geq \frac{\delta \tau \eps^2}{8} > 0, 
\end{align*}
which ends the proof of Theorem~\ref{thm:fail_second_moment}. $\qed$

\myskip 
\textbf{Remark --} Notice that a statement akin to Theorem~\ref{thm:fail_second_moment} might still hold even if $\Phi_d''(0) < 0$ for large $d$, as long 
as $\Phi_d$ reaches its global maximum in a value $q$ which is far from $0$ as $d \to \infty$, as our 
argument can then be easily adapted to this setting.
As such, we do not know if $\tau_\f(\kappa)$ (which 
comes out of our local analysis around $q = 0$) is a sharp threshold for $\EE[Z_\kappa^2] \gg \EE[Z_\kappa]^2$.

\myskip
\begin{proof}[Proof of Lemma~\ref{lemma:limit_Gsecond} --]
We start from eq.~\eqref{eq:d2G_dq2_1}, which for $q = 0$ gives:
\begin{align}\label{eq:G20}
    G_d''(0) &= \frac{d(d+1)}{2n} - \frac{d}{n} \EE[\Tr \bW^2] + \frac{d^2}{4n} \textrm{Var}[\Tr[\bW \bW']].
\end{align}
In eq.~\eqref{eq:G20}, $\bW$ and $\bW'$ are sampled independently according to the law 
$\bbP_\kappa$ of $\bZ \sim \GOE(d)$ conditioned on $\|\bZ\|_\op \leq \kappa$, i.e.\
for any test function $\varphi$:
\begin{align}\label{eq:law_conditioned_GOE}
    \EE_{\bbP_\kappa}[\varphi(\bZ)] &= \frac{\int \, \varphi(\bZ) \, \indi\{\|\bZ\|_\op \leq \kappa\} e^{-\frac{d}{4} \Tr[\bZ^2]} \rd \bZ}{\int \, \indi\{\|\bZ\|_\op \leq \kappa\} e^{-\frac{d}{4} \Tr[\bZ^2]} \rd \bZ}.
\end{align}
We know that for $\bW \sim \bbP_\kappa$, the empirical spectral distribution $\mu_\bW$ weakly converges (a.s.) to $\mu_\kappa^\star$ given by Theorem~\ref{thm:lsd_constrained_GOE}.
Since $\int \mu_\bW(\rd x) x^2 = \int \mu_\bW(\rd x) x^2 \indi\{|x| \leq \kappa\}$, we have by the Portmanteau theorem 
and dominated convergence: 
\begin{align}\label{eq:term_1_G20}
    \lim_{d \to \infty} \frac{1}{d}\EE[\Tr \bW^2] &= \int \mu_\kappa^\star(\rd x) \, x^2 \, \indi\{|x| \leq \kappa\} = \int \mu_\kappa^\star(\rd x) \, x^2.
\end{align}
We now focus on the last term of eq.~\eqref{eq:G20}.
Notice that $\EE[\Tr[\bW \bW']] = \Tr[(\EE \bW)^2] = 0$, since $\EE \bW = 0$ because $\bbP_\kappa$ is symmetric under $\bW \leftrightarrow - \bW$.
Moreover, for any orthogonal matrix $\bO \in \mcO(d)$, $\bW \deq \bO \bW \bO^\T$ (as is directly seen from eq.~\eqref{eq:law_conditioned_GOE}), so that we further have:
\begin{align}\label{eq:variance_Haar_measure}
   \textrm{Var}[\Tr[\bW \bW']] = \EE[\Tr[\bW \bW']^2]= \EE_{\bO, \bLambda, \bLambda'}[\Tr[\bO \bLambda \bO^\T \bLambda']^2].
\end{align}
In eq.~\eqref{eq:variance_Haar_measure}, $\bLambda = \Diag(\{\lambda_i\})$ is a diagonal matrix containing the eigenvalues of $\bW$ (and similarly for $\bLambda'$), 
and $\bO$ is an orthogonal matrix sampled from the Haar measure on $\mcO(d)$, independently of $\bW, \bW'$.
Thus:
\begin{align}\label{eq:variance_Haar_measure_2}
   \textrm{Var}[\Tr[\bW \bW']] = \sum_{i,j,k,l} \EE[\lambda_i \lambda_k] \EE[\lambda_j \lambda_l] \EE[O_{ij}^2 O_{kl}^2].
\end{align}
The terms involving $\lambda_i$ eq.~\eqref{eq:variance_Haar_measure_2} can be computed using the permutation invariance of the law of $\{\lambda_i\}$ as well as the invariance under $\bLambda \leftrightarrow - \bLambda$. 
Concretely, for all $i \in [d]$:
\begin{align}\label{eq:EE_lambdai2}
    \EE[\lambda_i^2] = \EE[\lambda_1^2] = \frac{1}{d} \sum_{j=1}^d \EE[\lambda_j^2] = \frac{1}{d} \EE[\Tr \bW^2],
\end{align}
and for $i \neq j$:
\begin{align}\label{eq:EE_lambdailambdaj}
    \EE[\lambda_i \lambda_j] &= \EE[\lambda_1 \lambda_2] = \frac{1}{d-1} \EE\left[\lambda_1 \sum_{k \geq 2} \lambda_k\right]
    = \frac{1}{d(d-1)} \EE[(\Tr \bW)^2 - \Tr(\bW^2)].
\end{align}
The first moments of the matrix elements of a Haar-sampled orthogonal matrix are elementary (see e.g.\ \cite{banica2011polynomial} for general results):
\begin{align}\label{eq:moments_Haar}
    \EE[O_{ij}^2 O_{kl}^2] &= \begin{dcases}
        \frac{3}{d(d+2)} &(i = k \textrm{ and } j = l), \\
        \frac{1}{d(d+2)} &(i = k \textrm{ and } j \neq l, \textrm{ or } i \neq k \textrm{ and } j = l), \\
        \frac{d+1}{d(d-1)(d+2)} &(i \neq k \textrm{ and } j \neq l).
    \end{dcases}
\end{align}
Using eq.~\eqref{eq:moments_Haar} in eq.~\eqref{eq:variance_Haar_measure_2}, separating 
cases in the sum, we get:
\begin{align}\label{eq:variance_Haar_measure_3}
    \nonumber
   \textrm{Var}[\Tr[\bW \bW']] &= 
   \frac{3}{d(d+2)} \cdot d^2 \cdot \EE[\lambda_1^2]^2 + 
   \frac{1}{d(d+1)} \cdot 2 d^2 (d-1) \cdot \EE[\lambda_1^2] \EE[\lambda_1 \lambda_2] \\ 
    \nonumber
   &+ \frac{d+1}{d(d-1)(d+2)} \cdot d^2(d-1)^2 \cdot \EE[\lambda_1 \lambda_2]^2, \\ 
   &= [1 +\smallO_d(1)] \left(3\EE[\lambda_1^2]^2 + 2d \EE[\lambda_1 \lambda_2] \EE[\lambda_1^2] + d^2 \EE[\lambda_1 \lambda_2]^2 \right).
\end{align}
From eqs.~\eqref{eq:term_1_G20} and \eqref{eq:EE_lambdai2}, we have $\EE[\lambda_1^2] \to \EE_{\mu_\kappa^\star}[X^2]$ as $d \to \infty$.
Furthermore, by Lemma~\ref{lemma:conc_moments_Pqkappa}, $\EE[(\Tr \bW)^2] = \Var[\Tr \bW] = \mcO(1)$ as $d \to \infty$, 
so eq.~\eqref{eq:EE_lambdailambdaj} gives that $d \EE[\lambda_1 \lambda_2] \to - \EE_{\mu_\kappa^\star}[X^2]$ as $d \to \infty$.
Plugging these limits in eq.~\eqref{eq:variance_Haar_measure_3} we get:
\begin{align}\label{eq:term_2_G20}
   \textrm{Var}[\Tr[\bW \bW']] &= 2 \left(\int \mu_\kappa^\star(\rd x) \, x^2 \right)^2 + \smallO_{d\to \infty}(1).
\end{align}
Finally, combining eqs.~\eqref{eq:G20}, \eqref{eq:term_1_G20} and \eqref{eq:term_2_G20} we obtain (recall $n/d^2 \to \tau$):
\begin{align}\label{eq:limit_G20}
    \lim_{d \to \infty} G_d''(0) &= \frac{1}{\tau} \left[\frac{1}{2} - \int \mu_\kappa^\star(\rd x) \, x^2 + \frac{1}{2}\left(- \frac{1}{2} \int \mu_\kappa^\star(\rd x) \, x^2\right)^2 \right].
\end{align}
The integral in eq.~\eqref{eq:limit_G20} was already computed in eq.~\eqref{eq:int_mukappa_x2}: plugging its value in eq.~\eqref{eq:limit_G20} 
shows that $  \lim_{d \to \infty} G_d''(0) = \tau_\f(\kappa)/\tau$, which ends the proof of Lemma~\ref{lemma:limit_Gsecond}.
\end{proof}

%% file: sections/appendix_freezing.tex
We discuss briefly here Open Problem~\ref{op:freezing}, more specifically how freezing could be established, 
or at least conjectured, from the second moment computation.
For the symmetric binary perceptron (SBP), this was the path followed in~\cite{aubin2019storage}, before
freezing was established rigorously in~\cite{perkins2021frozen,abbe2022proof}.

\myskip 
Let us first fix some notations.
For the remainder of this section we assume that the margin $\kappa \in (0,2)$ is given and fixed.
We denote $\Sigma_n \coloneqq \{\pm 1\}^n$, and $\mcS_d$ the set of symmetric $d \times d$ matrices.
For $\eps, \eps' \in \Sigma_n$, we denote $R(\eps, \eps') \coloneqq (1/n) \sum_{i=1}^n \eps_i \eps'_i$ their overlap.
We let $\bH \coloneqq (\bW_1, \cdots, \bW_n)$, and denote $S(\bH) \coloneqq \{\eps \in \{\pm 1\}^n \, \textrm{s.t.} \, \|\sum_{i=1}^n \eps_i \bW_i\|_{\op} \leq \kappa \sqrt{n} \}$.
Recall that $Z_\kappa = |S(\bH)|$, see eq.~\eqref{eq:def_Zkappa}.
In this appendix only we call $\bbP_0$ the law of $\bW_1, \cdots, \bW_n \iid \GOE(d)$, which was denoted $\bbP$ in the rest of the manuscript.

\begin{definition}[Freezing]\label{def:freezing}
    Let $\tau = n/d^2 = \Theta(1)$ be large enough so that $Z_\kappa \geq 1$ with probability $1 - \smallO(1)$ as $d \to \infty$ (i.e.\ we are in the SAT phase).
    The set $S(\bH)$ of solutions is said to exhibit freezing
    if there exists $q_c \in (0,1)$ such that the following holds:
    \begin{align*}
        \plim_{n,d \to \infty} \frac{1}{n} \log |\{\eps' \in S(\bH) \, : \, R(\eps, \eps') \geq q_c\}| = 0,
    \end{align*}
    where the limit is in probability over $\bH$ and $\eps \sim \Unif(S(\bH))$.
\end{definition}
\noindent
Definition~\ref{def:freezing} captures a relatively weak notion of freezing: for a typical solution $\eps$, the number of solutions with overlap to $\eps$ exceeding $q_c$ is at most sub-exponential in $n$. 
In the SBP, one can prove a stronger property, namely that $\eps$ is completely isolated and thus the \emph{unique} solution with such overlap. 
Below we briefly outline the two main ingredients in the classical approach to proving freezing---as carried out for the SBP---and explain why these steps present difficulties in the context of random matrix discrepancy. 
We refer to the literature cited above for more background and references.

\myskip 
\textbf{Step 1: contiguity --}
Freezing of solutions is a property of $\Unif(S(\bH))$, the uniform measure over the set of solutions.
The classical approach to establish properties of this measure is to use a contiguity argument with a \emph{planted} version of the problem, 
which is often easier to study.
Recall that, for two sequences of probability distributions $\bbP_d$ and $\bbQ_d$ (that depend on $d$), we say that $\bbP_d$ is contiguous to $\bbQ_d$, denoted $\bbP_d \lhd \bbQ_d$, 
if for any sequence of measurable events $A_d$:
\begin{align}\label{eq:contiguity}
(\bbQ_d(A_d) \to 0 \textrm{ as } d \to \infty)
 \Rightarrow 
(\bbP_d(A_d) \to 0 \textrm{ as } d \to \infty).
\end{align}
In our setting, we define the following two models.
\begin{itemize}[leftmargin=*]
    \item \textbf{Random model} ($\bbP_{\ra}$): draw first $\bW_1, \cdots, \bW_n \iid \GOE(d)$, conditioned on $S(\bH) \neq \emptyset$.
    Draw then $\eps \sim \Unif(S(\bH))$.
    \item \textbf{Planted model} ($\bbP_{\pl}$): draw first $\eps \sim \Unif(\Sigma_n)$. Draw then $\bW_1, \cdots, \bW_n \iid \GOE(d)$, conditioned on 
    satisfying $\|\sum_i \eps_i \bW_i\|_\op \leq \kappa \sqrt{n}$.
\end{itemize}
$\bbP_\ra$ and $\bbP_\pl$ define two probability measures over $\mcS_d^n \times \Sigma_n$. In the planted distribution, 
we often denote $\eps = \eps^\star$.
In many problems, $\bbP_\pl$ turns out to be easier to study than $\bbP_\ra$, and moreover we have the general identity, 
for any $(\eps, \bH)$ such that $\eps \in S(\bH)$:
\begin{align}\label{eq:planted_random}
    \frac{\rd \bbP_{\pl}}{\rd \bbP_{\ra}}(\eps, \bH) &= \frac{Z_\kappa(\bH)}{\EE_0[Z_\kappa(\bH')]} \cdot \bbP_0(S(\bH') \neq 0).
\end{align}
When $\bbP_0(S(\bH) \neq 0) \to 1$ as $d \to \infty$, contiguity of $\bbP_{\pl}$ and $\bbP_{\ra}$ thus follows from concentration of $Z_\kappa$ around its mean. 
More generally, suppose that for some non-decreasing sequence $\alpha_d > 0$ there exists $M > 0$ such that
\begin{align}\label{eq:concentration_fenergy}
    \lim_{d \to \infty}\bbP_0\!\left(\frac{Z_\kappa}{\EE[Z_\kappa]} \geq e^{-M\alpha_d}\right) = 1.
\end{align}
Then one can show that any event with probability $\exp(-\omega(\alpha_d))$ under $\bbP_{\pl}$ must have probability $\smallO(1)$ under $\bbP_\ra$~\citep{perkins2021frozen}.
In particular, if $\alpha_d = \mcO(1)$ then $\bbP_\pl \lhd \bbP_\ra$. 
Unfortunately, combining Theorems~\ref{thm:first_moment} and \ref{thm:second_moment} only implies a bound of the type of eq.~\eqref{eq:concentration_fenergy} 
for $\alpha_d = \Theta(d^2)$.
Establishing a better concentration bound for $Z_\kappa$ in average-case matrix discrepancy appears more delicate than in the SBP, 
where fluctuations of $Z_\kappa$ can be characterized exactly~\citep{abbe2022proof}. 
It is also unclear whether concentration should hold throughout the entire satisfiable phase, given the failure of the second-moment method in parts of the phase diagram (see Theorem~\ref{thm:fail_second_moment}). 
For these reasons, we leave the question of contiguity open. 
\begin{openquestion}[Contiguity]\label{op:contiguity}
    For $\kappa \in (0,2)$ and $\tau = n/d^2 = \Theta(1)$ sufficiently large (so that $S(\bH) \neq \emptyset$ with high probability), are $\bbP_\ra$ and $\bbP_\pl$ contiguous as $d \to \infty$? 
    Can one at least establish a bound of the form~\eqref{eq:concentration_fenergy} for some $\alpha_d = \smallO(d^2)$?
\end{openquestion}

\myskip
\textbf{Step 2: is the planted solutions isolated ?}
For $l \in [n]$, denote $q_l \coloneqq 2l/n - 1 \in [-1,1]$.
A key point is that for any $q_0 = q_{l_0} < q_1 = q_{l_1} \in [-1,1]$:
\begin{align}\label{eq:planted_sol_isolated}
    \nonumber
    &\EE_\pl \left[\# \{\eps \in S(\bH) \backslash \{\eps^\star\} \, : \, q_0 \leq R(\eps, \eps^\star) \leq q_1\}\right], \\ 
    \nonumber
    &\aleq \sum_{l_0 \leq l \leq l_1} \binom{n}{l} \bbP_{\pl} \left[\left\|\sum_{i=1}^n \eps_i \bW_i\right\|_\op \leq \kappa\right], \hspace{20pt} (\textrm{for any } \eps \in \Sigma_n \textrm{ s.t. } R(\eps, \eps^\star) = q_l), \\
    \nonumber
    &\bleq \sum_{l_0 \leq l \leq l_1}\binom{n}{l} \frac{2^n}{\EE_0[Z_\kappa]} \bbP_{0} \left[\|\bW\|_\op \leq \kappa \textrm{ and } \|q_l \bW + \sqrt{1-q_l^2} \bZ\|_\op \leq \kappa\right], \\
    &\cleq \sum_{l_0 \leq l \leq l_1}\binom{n}{l} \frac{\bbP_{0} \left[\|\bW\|_\op \leq \kappa \textrm{ and } \|q_l \bW + \sqrt{1-q_l^2} \bZ\|_\op \leq \kappa\right]}{\bbP_{0} \left[\|\bW\|_\op \leq \kappa\right]}.
\end{align}
where 
in $(\rm a)$ we used the invariance of the law of $\GOE(d)$, in $(\rm b)$ the definition of the planted distribution, and in $(\rm c)$ the expression of $\EE_0[Z_\kappa]$, see Section~\ref{sec:1st_moment}.
For a given $\tau = \lim n / d^2 = \Theta(1)$ large enough so that $Z_k \geq 1$ with high probability,
one can then 
study the asymptotics of the RHS of eq.~\eqref{eq:planted_sol_isolated}
by using
similar arguments to the ones developed in Section~\ref{sec:2nd_moment}.
Assuming contiguity, 
we deduce that a sufficient condition for the freezing of solutions 
(Definition~\ref{def:freezing})
is that there exists $q_c \in (0,1)$ such that for all $\eps \in (0, 1- q_c)$:
\begin{align}\label{eq:sufficient_cond_freezing}
    \limsup_{d \to \infty} \sup_{q \in (q_c, 1-\eps)} \gamma_d(q) < 0,
\end{align}
where
\begin{align}\label{eq:def_gammad}
   \gamma_d(q) \coloneqq \frac{1}{d^2} \log \frac{\bbP_0 \left[\|\bW\|_\op \leq \kappa \textrm{ and } \|q \bW + \sqrt{1-q^2} \bZ\|_\op \leq \kappa\right]}{\bbP_0 \left[\|\bW\|_\op \leq \kappa\right]} + \tau H\left(\frac{1+q}{2}\right).
\end{align}
Unfortunately, the bounds we derive in Section~\ref{sec:2nd_moment} for the first term in eq.~\eqref{eq:def_gammad} become trivial as $q$ approaches $1$. 
We therefore leave an investigation of eq.~\eqref{eq:sufficient_cond_freezing} as an open question, which could be illuminated by a sharp second moment analysis (see Open Problem~\ref{op:sharp_2nd_mom}).
\begin{openquestion}[Second moment potential close to $q = 1$]
    Establish whether eq.~\eqref{eq:sufficient_cond_freezing} holds for $\tau = n /d^2 = \Theta(1)$ sufficiently large (as a function of $\kappa \in (0,2)$), and some $q_c < 1$.
    A natural conjecture would be given by the behavior of $\gamma(q) \coloneqq \lim_{d \to \infty} \gamma_d(q)$ for $q$ close to $1$.
\end{openquestion}

%% file: sections/appendix_ldp.tex
For completeness, we provide here an alternative proof of Corollary~\ref{cor:ldp_Wop_var_principle}, which appeared in an earlier version of this manuscript,
and which relies only on the 
the main result of \cite{arous1997large}, that is a large deviation principle for the empirical eigenvalue distribution of $\GOE(d)$ matrices\footnote{
    Notice that~\cite{arous1997large} use a convention where $\GOE(d)$ matrices have off-diagonal entries with variance $1/(2d)$. We state their result adapted to our conventions.
}.
Recall that $\Sigma(\mu) \coloneqq \int \mu(\rd x) \mu(\rd y) \log |x-y|$.
\begin{proposition}[\cite{arous1997large}]\label{prop:ldp_emeasure}
    Let $\bW \sim \GOE(d)$. 
    We denote $\mu_{\bW} \coloneqq (1/d) \sum_{i=1}^d \delta_{\lambda_i(\bW)}$ its empirical spectral distribution. 
    For $\mu \in \mcM_1^+(\bbR)$, recall eq.~\eqref{eq:def_I}:
    \begin{equation*}
        I(\mu) \coloneqq -\frac{1}{2} \Sigma(\mu) + \frac{1}{4} \int \mu(\rd x) \, x^2 - \frac{3}{8}.
    \end{equation*}
    Then: 
    \begin{itemize}
        \item[$(i)$] $I : \mcM_1^+(\bbR) \to [0, \infty]$ is a strictly convex function and a good rate function, i.e.\  
        $\{I_1 \leq M\}$ is a compact subset of $\mcM_1^+(\bbR)$ for any $M > 0$.
        \item[$(ii)$]
        The law of $\mu_\bW$ satisfies a large deviation principle, in the scale $d^2$, with rate function $I$, that is for any 
        open (respectively closed) subset $O \subseteq \mcM_1^+(\bbR)$ (respectively $F \subseteq \mcM_1^+(\bbR)$): 
        \begin{equation*}
            \begin{dcases}
                \liminf_{d\to\infty} \frac{1}{d^2} \log \bbP[\mu_\bW \in O] \geq - \inf_{\mu \in O} I(\mu), \\
                \limsup_{d\to\infty} \frac{1}{d^2} \log \bbP[\mu_\bW \in F] \leq - \inf_{\mu \in F} I(\mu).
            \end{dcases}
        \end{equation*}
    \end{itemize}
\end{proposition}

\myskip
\textbf{Remark --}
\cite{arous1997large} also prove large deviations results for the empirical measure of $\bW$ sampled under more general distributions, with density 
proportional to $\exp\{-d\Tr[V(\bW)]/2\}$, for a continuous potential $V$ growing fast enough at infinity.
Here we essentially adapt these results to the discontinuous potential
\begin{equation*}
    V(x) = \frac{x^2}{2} + \indi\{|x| > \kappa\} \times \infty.
\end{equation*}

\subsection{Large deviation upper bound}\label{subsec:ldp_Wop_ub}

\myskip
Let $Q \coloneqq \{\mu \in \mcM_1^+(\bbR) \, : \, \mu([-\kappa, \kappa]) = 1\}$. 
By the Portmanteau theorem, $Q$ is sequentially closed under weak convergence, and thus closed since the weak topology on $\mcM_1^+(\bbR)$ is metrizable.
We apply Proposition~\ref{prop:ldp_emeasure} to get:
\begin{align}
    \label{eq:ub_Wop_var_principle}
    \nonumber
     \limsup_{d \to \infty} \frac{1}{d^2} \log \bbP[\|\bW\|_\op \leq \kappa] &=  \limsup_{d \to \infty} \frac{1}{d^2} \log \bbP[\mu_\bW \in Q] , \\ 
    \nonumber
     &\leq - \inf_{\mu \in Q} I(\mu), \\ 
     &= - \inf_{\mu \in \mcM_1^+([-\kappa,\kappa])} I(\mu).
\end{align}
The last equality follows since $I(\mu_{|[-\kappa,\kappa]}) = I(\mu)$ for all $\mu \in Q$.
This proves the upper bound for $(1/d^2) \log \bbP[\|\bW\|_\op \leq \kappa]$ in Corollary~\ref{cor:ldp_Wop_var_principle} 

\subsection{Large deviation lower bound}

We focus now on the lower bound for $(1/d^2) \log \bbP[\|\bW\|_\op \leq \kappa]$ in Corollary~\ref{cor:ldp_Wop_var_principle}.

\myskip
Unfortunately the large deviation statement of Proposition~\ref{prop:ldp_emeasure} is not enough to obtain the corresponding lower bound to eq.~\eqref{eq:ub_Wop_var_principle}, 
because the set of probability measures supported in $[-\kappa,\kappa]$ has empty interior under the weak topology\footnote{For any $\mu \in \mcM_1^+(\bbR)$, there is a sequence $\mu_n$ weakly converging to $\mu$ while $ \supp(\mu_n) \nsubseteq [-\kappa,\kappa]$.}. 
Instead, we come back to the joint law of the eigenvalues of a $\GOE(d)$ matrix, and 
restrict the integration domain to a small neighborhood of the quantiles of $\mu_\kappa^\star$.
This strategy is similar to the one used in the proof of the large deviation lower bound in \cite{arous1997large}.

\myskip 
We use the results of Section~\ref{subsec:var_principle_1st_mom}, which are independent of Corollary~\ref{cor:ldp_Wop_var_principle}, and characterize the minimizing measure of $I(\mu)$.
In particular, by Theorem~\ref{thm:properties_inf_I}, we know that $I(\mu)$ has a unique minimizer in $\mcM_1^+([-\kappa, \kappa])$, which we denote $\mu_\kappa^\star$, 
and it has a density $\rho_\kappa$ given in eq.~\eqref{eq:def_rhokappa}.
Let $\delta \in (0,\kappa)$.
In what follows, we let $\nu_\delta \coloneqq \mu^\star_{\kappa - \delta}$.
We define the quantiles of $\nu_\delta$ as 
\begin{equation*}
    -(\kappa-\delta) = x_0^{(d)} < x_1^{(d)} < \cdots < x_d^{(d)} < x_{d+1}^{(d)} = \kappa - \delta, 
\end{equation*}
with, for all $i \in \{0, \cdots, d\}$:
\begin{equation*}
    \int_{x_i^{(d)}}^{x_{i+1}^{(d)}} \nu_\delta(u) \, \rd u = \frac{1}{d+1}. 
\end{equation*}
We will drop the subscript and write $x_i$ for $x_i^{(d)}$ to lighten notations.

\myskip
Clearly, the empirical measure $(1/d)\sum_{i=1}^d \delta_{x_i}$ weakly converges to $\nu_\delta$ as $d \to \infty$.
Notice that $\{(\lambda_i)_{i=1}^d \, : \, |\lambda_i - x_i| \leq \delta, \, \, \forall i \in [1,d]\} \subseteq [-\kappa,\kappa]^d$, 
so that from eqs.~\eqref{eq:joint_law_evalues} and \eqref{eq:part_function_selberg}:
\begin{align}
    \label{eq:lb_Wop_1}
   &e^{-\frac{3d^2}{8} + \smallO(d^2)} \bbP[\|\bW\|_\op \leq \kappa] \\ 
    \nonumber
   &\geq \int_{[-\delta,\delta]^d} \prod_{i < j} |u_i - u_j + x_i - x_j| e^{-\frac{d}{4} \sum_{i=1}^d (u_i + x_i)^2} \prod_{i=1}^d \rd u_i,\\ 
    \nonumber
   &\ageq \int_{[-\delta,\delta]^d \cap \Delta_d} \prod_{i < j} (u_j - u_i + x_j - x_i) e^{-\frac{d}{4} \sum_{i=1}^d (u_i + x_i)^2} \prod_{i=1}^d \rd u_i, \\
    \nonumber
   &\bgeq \int_{[-\delta,\delta]^d \cap \Delta_d} \prod_{i + 1< j} (x_j - x_i) \prod_{i=1}^{d-1} [x_{i+1} - x_{i}]^{1/2} [u_{i+1} - u_i]^{1/2} e^{-\frac{d}{4} \sum_{i=1}^d (\delta + |x_i|)^2} \prod_{i=1}^d \rd u_i,
\end{align}
where we defined in $(\rm a)$ the set $\Delta_d \coloneqq \{u_1 < \cdots < u_d\}$, and used 
in $(\rm b)$ that $u_i \leq u_j$ and $x_i \leq x_j$, as well as the inequality $A+B \geq \sqrt{AB}$ for $A, B \geq 0$.
The integral on the variables $u_i$ can be lower-bounded as follows:
\begin{align}\label{eq:volume_u}
    \nonumber
    \int_{[-\delta,\delta]^d \cap \Delta_d} \prod_{i=1}^{d-1} \sqrt{u_{i+1} - u_i} \prod_{i=1}^d \rd u_i 
    &= \delta^{(3d-1)/2} \int_{[-1,1]^d \cap \Delta_d} \prod_{i=1}^{d-1} \sqrt{u_{i+1} - u_i} \prod_{i=1}^d \rd u_i , \\ 
    \nonumber
    &\geq \delta^{(3d-1)/2} \prod_{i=1}^{d} \int_{-1 + \frac{2(i-1)}{d}}^{-1 + \frac{2i-1}{d}} \rd u_i \, \left(\prod_{i=1}^{d-1} \sqrt{u_{i+1} - u_i}\right) , \\ 
    &\geq \left(\frac{\delta}{d}\right)^{(3d-1)/2}.
\end{align}
Combining eqs.~\eqref{eq:lb_Wop_1} and \eqref{eq:volume_u}: 
\begin{align*}
    \liminf_{d \to \infty} \frac{1}{d^2} \log \bbP[\|\bW\|_\op \leq \kappa]  
    &\geq \frac{3}{8} - \frac{\delta^2}{4}
    + \liminf_{d \to \infty} \left[-\frac{\delta}{2d} \sum_{i=1}^d |x_i| -\frac{1}{4d} \sum_{i=1}^d x_i^2\right. \\ 
    &\left.+ \frac{1}{2d^2} \sum_{i=1}^{d-1} \log (x_{i+1} - x_i) + \frac{1}{d^2} \sum_{i,j=1}^d \log(x_j - x_i) \indi\{j > i+1\}
    \right].
\end{align*}
By the weak convergence described above, and since $\sum_{i=1}^d x_i^2 = \sum_{i=1}^d x_i^2 \indi\{|x_i| \leq \kappa\}$, we get by the Portmanteau theorem:
\begin{align}\label{eq:lb_Wop_2}
    \liminf_{d \to \infty} \frac{1}{d^2} \log \bbP[\|\bW\|_\op \leq \kappa]  
    &\geq \frac{3}{8} - \frac{\delta^2}{4} - \frac{\delta}{2} \int |x| \nu_\delta(\rd x) - \frac{1}{4} \int x^2 \nu_\delta(\rd x) \\
    \nonumber
    &\hspace{-9pt}+ \liminf_{d \to \infty}\left[\frac{1}{2d^2} \sum_{i=1}^{d-1} \log (x_{i+1} - x_i) + \frac{1}{d^2} \sum_{i,j=1}^d \log(x_j - x_i) \indi\{j > i+1\}
    \right].
\end{align}
Finally, notice that 
\begin{align*}
    \Sigma(\nu_\delta) &= 2 \sum_{0 \leq i,j \leq d} \int_{x_i}^{x_{i+1}} \nu_\delta(\rd x) \int_{x_j}^{x_{j+1}} \nu_\delta(\rd y) \, \log (y-x) \, \indi\{x < y\}, \\ 
    &\leq \sum_{i=0}^d \int_{x_i}^{x_{i+1}} \nu_\delta(\rd x) \int_{x_i}^{x_{i+1}} \nu_\delta(\rd y) \, \log |y-x| + 2 \sum_{0 \leq i < j \leq d} \frac{\log(x_{j+1} - x_i)}{(d+1)^2}, \\
    &\leq \frac{1}{(d+1)^2} \left[\sum_{i=0}^d \log(x_{i+1}-x_i) + 2 \sum_{0 \leq i < j \leq d} \log(x_{j+1} - x_i)\right], \\ 
    &\aleq \frac{1}{(d+1)^2} \left[\sum_{i=1}^{d-1} \log(x_{i+1}-x_i) + 2 \sum_{i,j=1}^d \log(x_j - x_i) \indi\{j > i+1\} \right] + \frac{2d+1}{(d+1)^2} \log 2 \kappa.
\end{align*}
We used in $(\rm a)$ that $|x_i - x_j| \leq 2 (\kappa-\delta) \leq 2\kappa$ for all $i,j$.
Using this in eq.~\eqref{eq:lb_Wop_2} gives:
\begin{align*}
    \liminf_{d \to \infty} \frac{1}{d^2} \log \bbP[\|\bW\|_\op \leq \kappa]  
    &\geq \frac{3}{8} - \frac{\delta^2}{4} - \frac{\delta}{2} \int |x| \nu_\delta(\rd x) - \frac{1}{4} \int x^2 \nu_\delta(\rd x) + \frac{1}{2} \Sigma(\nu_\delta), \\
    &\geq - \frac{\delta}{2} \int |x| \nu_\delta(\rd x)  - \frac{\delta^2}{4} - I(\nu_\delta).
\end{align*}
Recall that $\nu_\delta = \mu^\star_{\kappa - \delta}$, so that taking the limit $\delta \to 0$, we get:
\begin{align}\label{eq:lb_Wop_var_principle}
    \liminf_{d \to \infty} \frac{1}{d^2} \log\bbP[\|\bW\|_\op \leq \kappa] 
    &\geq - I(\mu_\kappa^\star) = - \inf_{\mu \in \mcM_1^+([-\kappa,\kappa])} I(\mu).
\end{align}
Combining eqs.~\eqref{eq:ub_Wop_var_principle} and \eqref{eq:lb_Wop_var_principle} ends the proof of Corollary~\ref{cor:ldp_Wop_var_principle}.
$\qed$